\documentclass[12pt,a4paper,reqno]{amsart}
\usepackage{amscd}
\usepackage{amssymb}
\usepackage{amsthm}
\usepackage{graphicx}
\usepackage{color}
\usepackage[all]{xy}
\usepackage{mathrsfs}
\usepackage{tikz-cd}
\usepackage{mathtools}
\usepackage{hyperref}

 \newtheorem{theorem}{Theorem}[section]
\newtheorem{lemma}[theorem]{Lemma}

\newtheorem{corollary}[theorem]{Corollary}

\newtheoremstyle{defstyle}
  {.6em} 
  {.1em} 
  {} 
  {} 
  {\bfseries} 
  {.} 
  {.5em} 
  {} 
\theoremstyle{defstyle} \newtheorem{definition}[theorem]{Definition}

\newtheorem{example}[theorem]{Example}

\theoremstyle{remark}
\newtheorem{remark}[theorem]{Remark}

\numberwithin{equation}{section}
\numberwithin{figure}{section}

\newcommand{\NN} {\mathbb{N}}

\newcommand{\QQ} {\mathbb{Q}}
\newcommand{\RR} {\mathbb{R}}
\newcommand{\CC} {\mathbb{C}}

\newcommand {\foD}  {\mathfrak{D}}
\newcommand {\fod}  {\mathfrak{d}}

\newcommand {\coker} {\operatorname{coker}}
\newcommand {\Hom}  {\operatorname{Hom}}

\newcommand {\Int}  {\operatorname{Int}}
\renewcommand {\ker } {\operatorname{ker}}
\newcommand {\kk} {\Bbbk}
\newcommand {\scrP}  {\mathscr{P}}

\newcommand {\Sing} {\operatorname{Sing}}

\newcommand {\Supp} {\operatorname{Supp}}

\def\virt{\mathrm{vir}}
\def\PP{\mathbb{P}}

\def\Z{\mathbb{Z}}

\def\Q{\mathbb{Q}}

\def\G{\mathsf{G}}

\begin{document}

\title[Quivers and curves in higher dimension]{Quivers and curves in higher dimension}

\author[H.\,Arg\"uz]{H\"ulya Arg\"uz}
\address{University of Georgia, Department of Mathematics, Athens, GA 30605}
\email{Hulya.Arguz@uga.edu}

\author[P.\,Bousseau]{Pierrick Bousseau}
\address{University of Georgia, Department of Mathematics, Athens, GA 30605}
\email{Pierrick.Bousseau@uga.edu}

\date{}

\begin{abstract}
We prove a correspondence between Donaldson--Thomas invariants of quivers with potential having trivial attractor invariants and genus zero punctured Gromov--Witten invariants of holomorphic symplectic cluster varieties. The proof relies on the comparison of the stability scattering diagram, describing the wall-crossing behavior of Donaldson--Thomas invariants, with a scattering diagram capturing punctured Gromov--Witten invariants via tropical geometry.
\end{abstract}

\maketitle

\setcounter{tocdepth}{1}
\tableofcontents
\setcounter{section}{-1}
\section{Introduction}

\subsection{Overview}
Donaldson--Thomas (DT) invariants are counts of stable objects in a triangulated category $\mathcal{C}$ which is Calabi--Yau of dimension 3  \cite{donaldson1998gauge, JoyceSong, kontsevich2008stability, MR1818182}, and have important applications in physics, in geometry and in representation theory. In  physics, in particular in quantum field theory and string theory, DT invariants play an important role as counts of BPS states and D-branes \cite{MR2567952}. From the geometric point of view, particularly interesting situations are when $\mathcal{C}$ is the derived category of coherent sheaves \cite{MR1818182} or the Fukaya category of a Calabi--Yau 3-fold \cite{MR1941627, MR1957663}. In the context of representation theory, quivers with potentials \cite{MR2480710} provide a natural source of examples of Calabi--Yau categories of dimension three \cite{ginzburg2006calabi-yau, MR2484733}.
Due to its more algebraic nature, DT theory of quivers with potentials provides an ideal framework to study and explore many questions, which are also of interest in the geometric counterpart of DT theory.

Stable objects in a triangulated category are defined with respect to a Bridgeland stability condition on this category \cite{MR2373143}, and 
DT invariants are piecewise-constant with respect to the choice of stability condition:
they are constant in the complement of
countably many real codimension one loci in the space of stability conditions called walls, but they
jump discontinuously in general when the stability condition crosses a wall. The jumps of DT invariants across walls in the space of stability conditions are given by a universal wall-crossing formula due to Joyce--Song \cite{JoyceSong}
and 
Kontsevich--Soibelman \cite{kontsevich2008stability}.

Remarkably, a very similar wall-crossing formula appears in the a priori very different context of counting holomorphic disks in mirror symmetry \cite{auroux2007mirror,GSannals,kontsevich2008stability}. This observation leads to the surprising expectation that in many cases DT invariants may be equal to counts of holomorphic curves, such as log and punctured Gromov--Witten invariants of Abramovich--Chen--Gross--Siebert \cite{logGWbyAC, ACGS, logGW}, which appear in the algebro-geometric mirror construction of Gross--Siebert \cite{gross2021canonical}. Therefore, one naturally expects DT invariants to be related in some situations to counts of holomorphic curves in holomorphic symplectic varieties \cite{bousseau2022holomorphic,KS, lu2010instanton}. 
We show that this expectation holds in the context of DT invariants of quivers with potentials. In particular, we prove a correspondence between quiver DT invariants and punctured log Gromov--Witten invariants of holomorphic symplectic cluster varieties.

\subsection{Background and main result}
We state our main result after a quick review of quiver DT and punctured Gromov--Witten invariants.
\subsubsection{Quiver DT Invariants}
A quiver with potential $(Q,W)$ is given by a finite oriented graph $Q$, and a finite formal linear combination $W$ of oriented cycles in $Q$. 
Given a quiver with potential $(Q,W)$, with set of vertices $Q_0$, one can define 
a DT invariant $\Omega_\gamma^{+,\theta} \in \Z$ for every dimension vector $\gamma \in N_Q:=\bigoplus_{i\in Q_0}\Z s_i$ and every stability parameter \[ \theta \in \gamma^\perp \subset M_{Q,\RR}:=\Hom(N_Q,\RR)\,,\] as reviewed in \S\ref{sec_quiver_dt}. The dependence of $\Omega_\gamma^{+,\theta}$ on the stability parameter $\theta$ is captured by a universal wall-crossing formula \cite{JoyceSong,kontsevich2008stability}. 
It follows from the wall-crossing formula that general DT invariants $\Omega_\gamma^{+,\theta}$ are determined by particular DT invariants $\Omega_\gamma^{+,\star}$ called \emph{attractor DT invariants}, which are defined for a stability parameter $\theta$ close to the attractor point $\iota_\gamma \omega_Q =\omega_Q(\gamma, -) \in \gamma^\perp$, where $\omega_Q$ is the skew-symmetric form on $N_Q$ obtained by skew-symmetrization of the Euler form -- see \eqref{Eq: Euler form} \cite{AlexandrovPioline, KS, mozgovoy2020attractor}.

When $\Omega_{s_i}^{+,\star}=1$ for all $i\in Q_0$, and $\Omega_\gamma^{+,\star}=0$ unless $\gamma=s_i$ for some $i$ or $\gamma \in \ker \omega_Q$, we say that $(Q,W)$ has \emph{trivial attractor DT invariants}. 
As reviewed in Examples \ref{ex_1}-\ref{ex_2}, this condition is known to hold for many quivers with potentials of interest in representation theory and geometry. It is also expected to be a general property of quivers with potentials describing the derived category of coherent sheaves on non-compact Calabi--Yau 3-folds.

Throughout this paper we frequently consider rational DT invariants $\overline{\Omega}_\gamma^{+,\theta}$, which are a repackaging of the integer DT invariants $\Omega_\gamma^{+,\theta}$ defined by a universal formula -- see \eqref{eq_dt_rational} for details.

\subsubsection{Punctured Gromov--Witten Invariants}
Given a pair $(X,D)$ consisting of a smooth projective variety $X$ over $\mathbb{C}$ and a normal crossing divisor on $X$, punctured Gromov--Witten invariants of $(X,D)$ are virtual counts of curves in $X$ with prescribed tangency conditions along $D$ \cite{ACGS}. In general, such curves might have components contained in $D$ and logarithmic geometry is needed to make sense of the contact orders with $D$ \cite{logGWbyAC,logGW}. For general log Calabi-Yau pairs $(X,D)$, Gross--Siebert define particular genus zero punctured Gromov--Witten invariants, which play an important role in their general construction of mirrors, and which can be viewed as an algebro-geometric definition of counts of Maslov index zero holomorphic disks in the non-compact Calabi-Yau variety $U=Y\setminus D$ \cite{gross2021canonical}. 

We consider these invariants for particular log Calabi-Yau pairs $(X,D)$. A \emph{symplectic seed} $\mathbf{s}=(N,(e_i)_{i\in I},\omega)$ consists of a finite rank abelian group $N$, finitely many elements $e_i \in N$, and a skew-symmetric form $\omega$ on $N$, which is non-degenerate over $\Q$. We further assume that $v_i :=\iota_{e_i}\omega \in M$ is primitive -- see \eqref{eq_cluster_assumption}, where $M= \mathrm{Hom}(N,\mathbb{Z})$ is dual abelian group. Given a symplectic seed $\mathbf{s}$, one first considers a toric variety $X_\Sigma$ with fan $\Sigma$ in $M_{\RR} = M \otimes_{\mathbb{Z}} \RR = \mathrm{Hom}(M,\RR)$ containing the rays $\RR_{\geq 0}v_i$, and then one defines a pair $(X,D)$, where $X$ is the blow up of $X_\Sigma$ along the loci of equation $1+z^{e_i}$ in the toric divisor corresponding to the ray $\RR_{\geq 0}  v_i$, and where $D$ is the strict transform of the toric boundary divisor of $X_\Sigma$. Under an appropriate assumption -- see \eqref{eq_H_int}, $X$ is smooth, and so $(X,D)$ is a log Calabi-Yau pair. Moreover, $\omega$ induces an holomorphic symplectic form on the complement $U=Y \setminus D$. We refer to $U$ as the \emph{cluster variety} defined by $\mathbf{s}$, and to $(X,D)$ as a \emph{log Calabi-Yau compactification of the cluster variety}.

Following \cite{gross2021canonical}, one can define a genus zero punctured Gromov--Witten invariant $N_{\tau,\beta}^{(X,D)} \in \Q$ of $(X,D)$ for every curve class $\beta$ and every combinatorial choice of a so-called \emph{wall type}. 
The invariant $N_{\tau,\beta}^{(X,D)}$ is a virtual count of rational curves in $(X,D)$ of class $\beta$ and whose combinatorics of intersections with the strata of $D$ is constrained by $\tau$ -- see \S \ref{sec_gw} for details.

\subsubsection{Main result}
We prove a correspondence between the DT invariants $\Omega_\gamma^{+,\theta}$ of a quiver with potential $(Q,W)$
having trivial attractor DT invariants, and the punctured Gromov--Witten invariants $N_{\tau,\beta}^{(X,D)}$ of a log Calabi-Yau compactification $(X,D)$ of the cluster variety defined by a symplectic seed $\mathbf{s}$ ``compatible'' with $Q$. Here, the compatibility 
between a quiver $Q$ and a symplectic seed $\mathbf{s}=(N,(e_i)_{i\in I},\omega)$ is the data of linear map $\psi: N_Q \rightarrow N$ such that $\psi(s_i)=e_i$, $\omega_Q=\psi^\star \omega$, and 
$\psi \otimes \Q$ is surjective -- see \S \ref{sec_compatibility} for details. If $Q$ and $\mathbf{s}$ are compatible, then the holomorphic symplectic cluster variety $U=X \setminus D$ is a finite quotient of a symplectic fiber of the Poisson $\mathcal{X}$ cluster variety defined by $Q$ \cite{FG2, GHKbirational} -- see Remark \ref{remark_finite_group} for details. In particular, the dimension of $U$ and $X$ is equal to the rank of the skew-symmetrized Euler form $\omega_Q$ of $Q$.
Our main result, Theorem \ref{thm_main}, states:

\begin{theorem} \label{thm_main_intro}
Let $(Q,W)$ be a quiver with potential having trivial attractor DT invariants, $\mathbf{s}=(N,(e_i)_{i\in I}, \omega)$ a symplectic seed, and $(X,D)$ a log Calabi-Yau compactification of the corresponding cluster variety satisfying assumptions \eqref{eq_cluster_assumption} and \eqref{eq_H_int}. Fix a compatibility data $\psi \colon N_Q \rightarrow N$ between $\mathbf{s}$ and $Q$ and let $\gamma \in N_Q \setminus \ker \omega_Q$ be a dimension vector such that $\gamma \notin \Z_{\geq 1}s_i$ for all $i \in I$. Denote by $\psi^{\vee}: M_{\RR} \to M_{Q,\RR}$ the induced map between the dual vector spaces. Then, for every general stability parameter $\theta \in \psi^\vee (M_\RR) \cap \gamma^{\perp}\subset M_{Q,\RR}$ and point $x \in M_\RR$ such that $\psi^\vee(x)=\theta$, there exists a set of wall types $\mathcal{T}_\gamma^x$ and a curve class $\beta_\gamma^x$ described in \S\ref{sec_dt_gw}, such that we have the correspondence
\begin{equation} \label{eq_main_intro} 
\overline{\Omega}_\gamma^{+,\theta} 
= \frac{1}{|\gamma|}\sum_{\tau \in \mathcal{T}_\gamma^x} k_\tau N_{\tau,\beta_\gamma^x}^{(X,D)} \,,\end{equation}
between the rational DT invariants $\overline{\Omega}_\gamma^{+,\theta}$ of $(Q,W)$, and the genus zero punctured Gromov--Witten invariants 
   $N_{\tau,\beta_\gamma^x}^{(X,D)}$ of $(X,D)$, where $|\gamma|$ is the divisibility of $\gamma$ in $N_Q$,
   and
   the coefficient $k_\tau$ is a positive integer defined in \eqref{eq_coeff}.
\end{theorem}

The proof of Theorem \ref{thm_main_intro} is based on a comparison between the stability scattering diagram for DT invariants \cite{Bridgeland} and the explicit description, due to Mark Gross and the first author \cite{HDTV}, of the canonical scattering diagram capturing punctured Gromov--Witten invariants of log Calabi-Yau pairs obtained as blow-ups of toric varieties.

Using Theorem \ref{thm_main_intro},
    one can interpret the integer DT invariants 
    $\Omega_\gamma^{+,\theta}$ as BPS invariants underlying the rational punctured Gromov--Witten invariants $N_{\tau,\beta}^{(X,D)}$. 
    Under the assumptions of Theorem \ref{thm_main_intro}, we will also show that the DT invariants $\Omega_\gamma^{+,\theta}$
are non-negative integers -- see Theorem \ref{thm_nice_dt}. Both the integrality and the positivity of $\Omega_\gamma^{+,\theta}$ are highly non-trivial from the point of view of punctured Gromov--Witten theory.

While Theorem \ref{thm_main_intro} is about DT invariants of quivers with potentials, it can sometimes be applied to geometrically defined DT invariants when the derived category of coherent sheaves of a Calabi-Yau 3-fold admits a description in terms of a quiver with potential. We give an explicit example of such an application to the DT invariants of local $\PP^2$ in Theorem \ref{thm_localP2}.

Theorem \ref{thm_main_intro} relates all DT invariants of a quiver with potential, and so of a Calabi-Yau category of dimension 3, to punctured Gromov--Witten invariants of a log Calabi--Yau compactification of a holomorphic symplectic cluster variety of dimension equal to the rank of the skew-symmetrized Euler form. In particular, this result is very different from the MNOP correspondence \cite{MNOP1} which relates DT counts of ideal sheaves and Gromov--Witten invariants of the same 3-fold.

\subsection{Related works}

\subsubsection*{The tropical vertex} 
The first example of the general correspondence given in Theorem \ref{thm_main_intro} was obtained by Gross--Pandharipande \cite{gross2010quivers}, following previous work of Gross--Pandharipande--Siebert on the tropical vertex \cite{GPS}: they proved a correspondence between DT invariants of the $m$-Kronecker quivers, consisting of two vertices connected by $m$ arrows, and log Gromov--Witten invariants of log Calabi--Yau surfaces. 
This correspondence was generalized by Reineke--Weist to the case of complete bipartite quivers under the name of refined Gromov--Witten/Kronecker correspondence \cite{MR3033514, MR3004575}, and then to arbitrary acyclic quivers with skew-symmetrized Euler forms of rank two by the second author \cite[\S 8.5]{bousseau2020quantum}. Theorem \ref{thm_main_intro} generalize these results to higher dimension. Actually, even in dimension two, Theorem \ref{thm_main_intro} is broader than previously known results as it can be applied to non-acyclic quivers of rank two having trivial attractor DT invariants. We give examples of such application in 
\S \ref{sec_ex_local}-\ref{sec_ex_cubic}. 
There also exist a refined DT/ higher genus GW generalization of the Gromov--Witten/Kronecker correspondence \cite[\S 8.5]{bousseau2020quantum} and extensions to different geometries in dimension two \cite{bousseau2018example,reineke2021moduli},  and it is an interesting question to find out if they admit higher dimensional generalizations.

Finally, there should exist analogues of Theorem \ref{thm_main_intro} in the context of geometric DT invariants of non-compact Calabi--Yau 3-folds. Such a correspondence is proved in \cite{Bp2} for geometric DT invariants of the local projective plane. 

\subsubsection*{Knots-quivers correspondence} An a priori different relation between quiver DT invariants and open Gromov--Witten theory has been explored in the context of the knots-quivers correspondence \cite{MR4156213}. This correspondence involves DT invariants of symmetric quivers, and so in particular with zero skew-symmetrized Euler form, whereas Theorem \ref{thm_main_intro} deals with the complementary case of dimension vectors which are not in the kernel of the skew-symmetrized Euler form. 
However, the symmetric quivers constructed in the knots-quiver correspondence appear to share the same geometric origin as the quivers of the present article, the former being built from basic disks and linking numbers between them, the latter from disks created from non-toric blow-ups.
It is a very interesting question to understand what is the precise relation between Theorem \ref{thm_main_intro} and the correspondence in \cite{MR4156213}.

\subsubsection*{Quivers, flow trees and log curves}
A correspondence between quiver DT invariants and counts of tropical curves has been established by Cheung--Mandel \cite{mandel2020disks} and by the authors in their proof of the flow tree formula expressing general DT invariants in terms of attractor DT invariants \cite{ABflow}.
Using a tropical/log Gromov--Witten correspondence, it was then proved by the authors that the universal coefficients appearing in the flow tree formula are log Gromov--Witten invariants of toric varieties \cite{ABflowquiver}. The toric varieties considered in \cite{ABflowquiver} have dimension equal to the number of vertices of $Q$, but it follows from the tropical/log Gromov--Witten correspondence that the same log Gromov--Witten invariants can be obtained from toric varieties of dimension equal to the rank of $\omega_Q$, which is also the dimension of the cluster varieties considered in this paper.
For quivers with potentials having trivial attractor DT invariants, the compatibility between Theorem \ref{thm_main_intro} and the main result of \cite{ABflowquiver} should have a geometric interpretation as a degeneration formula in Gromov--Witten theory for the degeneration of the log Calabi--Yau compactification of the cluster variety to a toric variety and other simpler pieces, as studied in 
\cite{HDTV,GPS}. In the two-dimensional case, this degeneration formula is studied in \cite{GPS} and its interpretation in terms of quiver DT invariants is discussed in \cite{MR3033514}.

\subsection{Acknowledgments} 
The research of Hülya Argüz was partially supported by the NSF grant DMS-2302116. The research of Pierrick Bousseau was partially supported by the NSF grant DMS-2302117. Final parts of this paper were completed during the ``Inaugural Simons Math Summer Workshop'' at the Simons Center for Geometry and Physics, organized by Mark Gross and Mark McLean. We also thank the anonymous referee for their careful reading and the many suggestions to improve the exposition.

\section{Donaldson--Thomas invariants of quivers}

\subsection{Quiver DT invariants} 
\label{sec_quiver_dt}

A \emph{quiver} $Q$ is a finite oriented graph. We denote by 
$Q_0$ the set of vertices of $Q$, and $Q_1$ the set of oriented edges of $Q$, referred to as arrows.  
We set
\[N_Q:=\Z^{Q_0}=\bigoplus_{i \in Q_0}\Z s_i\,.\]
We denote by $M_Q :=\Hom (N_Q,\Z)$ the dual lattice to $N_Q$, and  
\[ M_{Q,\RR}:= M_Q \otimes_\Z \RR =\Hom (N_Q,\RR)\,, \]
the associated dual vector space.

\begin{definition}
A \emph{representation} of a quiver $Q$, denoted by
\[V=(\{V_i\}_{i\in Q_0},\{f_{\alpha}\}_{\alpha \in Q_1}) \,,\]
is an assignment of a finite-dimensional vector space $V_i$ over $\CC$ for each vertex $i \in Q_0$, and a $\CC$-linear map $f_{\alpha} \in \mathrm{Hom}(V_i,V_j)$ for each arrow $(\alpha:
i \rightarrow j) \in Q_1$. The \emph{dimension vector} associated to a quiver representation is the vector
\[\gamma = (\gamma_i)_{i \in Q_0} \in N_Q\,,\] 
where $\gamma_i :=\dim V_i$.
\end{definition}

We have the following notion of stability due to King \cite{MR1315461}.
\begin{definition}[King's stability]
Let $V$ be a quiver representation with associated dimension vector $\gamma \in N_Q$. A \emph{stability parameter} for $\gamma$ is a point \[ \theta \in \gamma^{\perp}:=\{ \theta \in M_{Q,\RR}\,, \theta(\gamma)=0\} \subset M_{Q,\RR}\,.\] 
The representation $V$ is $\theta$-\emph{stable} (resp.\ $\theta$-\emph{semistable}) if for all 
non-zero strict subrepresentation $V'$ of $V$ we have $\theta(\mathrm{dim}(V')) < 0 $ (resp.\ $\theta(\mathrm{dim}(V')) \leq 0 $).
\end{definition}

For a dimension vector $\gamma \in N_Q$, we say a stability parameter $\theta \in \gamma^\perp$ is \emph{$\gamma$-general} if $\theta(\gamma')=0$ implies $\gamma'$ collinear with $\gamma$. For $\gamma \in N_Q$ and a  $\gamma$-general stability parameter $\theta \in M_{Q,\RR}$, the moduli space $\mathcal{M}_{\gamma}^{\theta}$ of S-equivalence classes of $\theta$-semistable quiver representations of $Q$ dimension $\gamma$  is a quasi-projective variety over $\CC$, 
which is constructed via geometric representation theory -- see \cite{MR1315461}.  When $Q$ is acyclic, $\mathcal{M}_{\gamma}^{\theta}$ is projective. In general, the choice of a \emph{potential} $W=\sum_c \lambda_c c$, that is, of a finite linear combination of oriented cycles $c$ in $Q$ and with coefficients $\lambda_c \in \CC$, defines a \emph{trace function} $ \mathrm{Tr}(W)_{\gamma}^{\theta}:  \mathcal{M}_{\gamma}^{\theta} \to \CC $ on each of the moduli spaces 
$\mathcal{M}_{\gamma}^{\theta}$: if $c$ is an oriented cycle of arrows $\alpha_1,\dots,\alpha_n$ of $Q$, we define
\begin{align}
\nonumber
\mathrm{Tr}(c)_{\gamma}^{\theta}:  \mathcal{M}_{\gamma}^{\theta} & \longrightarrow \CC \\ 
\nonumber
     V = (V_i, f_{\alpha}) & \longmapsto \mathrm{Tr}(f_{\alpha_n} \circ  \ldots \circ f_{\alpha_1} )
\nonumber
\end{align}
and we set
$$\mathrm{Tr}(W)_{\gamma}^{\theta}  = \sum_{c} \lambda_c \mathrm{Tr}(c)_{\gamma}^{\theta} \,\ .$$

Quiver DT invariants of a quiver with potential $(Q,W)$ are then defined as follows. Fix a dimension vector $\gamma \in N_Q$, and a 
$\gamma$-general stability parameter $\theta \in \gamma^\perp$. The DT (Donaldson--Thomas) invariant, denoted by \[\Omega_\gamma^{+,\theta} \in \Z,\] 
is an integer corresponding to the virtual count of the critical points of the trace function $\mathrm{Tr} (W)_\gamma^\theta$ on the moduli space $\mathcal{M}_{\gamma}^{\theta}$ of $\theta$-semistable representations of dimension $\gamma$. If the $\theta$-stable locus in $\mathcal{M}_{\gamma}^{\theta}$ is empty, we have $\Omega_\gamma^{+,\theta}=0$.
Else, $\Omega_\gamma^{+,\theta}$ is 
defined as
\begin{equation} \label{eq_dt}
\Omega_\gamma^{+,\theta}
=e(\mathcal{M}_\gamma^\theta, \phi_{\mathrm{Tr}(W)_\gamma^\theta}(IC))
=\sum_{i}(-1)^{i} \dim  H^i(\mathcal{M}_{\gamma}^{\theta},  \phi_{\mathrm{Tr}(W)_\gamma^\theta}(IC))\in \Z
\,,\end{equation}
where $e(-)$ is the Euler characteristic, $IC$ is the intersection cohomology sheaf on $\mathcal{M}_{\gamma}^{\theta}$ normalized to be the constant sheaf in degree $0$ if $\mathcal{M}_\gamma^\theta$ is smooth, and $\phi_{\mathrm{Tr}(W)_\gamma^\theta}$ is the vanishing cycle functor defined by the trace function
\cite{davison2015donaldson, davison2016cohomological, MR4000572}.

For every $\gamma \in N_Q$, we set 
\begin{equation}\label{eq_kappa}
\kappa(\gamma):=(-1)^{\chi_Q(\gamma,\gamma)} \in \{\pm 1\}\,,\end{equation}
where $\chi_Q: N_Q\times N_Q \rightarrow \Z$ is the Euler form of $Q$, given by 
\begin{equation} \label{eq_euler}
\chi_Q(\gamma,\gamma')=\sum_{i \in Q_0} \gamma_i \gamma_i' - \sum_{(\alpha:i \rightarrow j)\in Q_1} \gamma_i \gamma_j' \,.
\end{equation}
The data of the DT invariants 
$\Omega_\gamma^{+,\theta}$ can be repackaged into the \emph{rational DT invariants} 
\begin{equation}
\label{eq_dt_rational}
\overline{\Omega}_\gamma^{+,\theta} := \sum_{\substack{\gamma' \in N_Q\\ 
    \gamma=k \gamma',\, k\in \Z_{\geq 1}}} \frac{\kappa(\gamma')^{k-1}}{k^2} \Omega_{\gamma'}^{+,\theta} \in \QQ \,,
\end{equation}
which are often more convenient to work with, particularly when calculating them using wall structures \cite{ABflow}. Rational DT invariants can also be defined using the motivic Hall algebra \cite{JoyceSong, kontsevich2008stability, MR2650811,MR2801406}.

\begin{remark}
In the literature on DT invariants, it is more common to work with DT invariants $\Omega_\gamma^{\theta}$ defined by 
\[
\Omega_\gamma^{\theta}
:=(-1)^{\dim \mathcal{M}_\gamma^\theta}\, e(\mathcal{M}_\gamma^\theta, \phi_{\mathrm{Tr}(W)_\gamma^\theta}(IC)) \in \Z\,,\]
when the $\theta$-stable locus is non-empty, 
and the rational DT invariants $\overline{\Omega}_\gamma^\theta$ defined 
by 
\begin{equation} \nonumber
\overline{\Omega}_\gamma^\theta :=
\sum_{\substack{\gamma' \in N_Q\\ \gamma=k \gamma', k\in \Z_{\geq 1}}} \frac{1}{k^2} \Omega_{\gamma'}^\theta \in \QQ\,. \end{equation}
When the $\theta$-stable locus is non-empty, we have $\dim \mathcal{M}_\gamma^\theta = 1-\chi_Q(\gamma, \gamma)$, and so we have 
\[ \Omega_\gamma^{+,\theta}=-\kappa(\gamma) \Omega_\gamma^\theta\]
and one can check that
\[ \overline{\Omega}_\gamma^{+,\theta}=-\kappa(\gamma) \overline{\Omega}_\gamma^\theta\,.\]
In this paper, we work with the DT invariants $\Omega_\gamma^{+,\theta}$ as they have better positivity properties -- see Theorem \ref{thm_nice_dt}.
\end{remark}

While quiver DT invariants generally depend on the choice of the potential since the trace of $W$ appears in \eqref{eq_dt}, we will focus attention on particular quivers which have trivial \emph{attractor DT invariants} discussed in the following section -- for these quivers and dimension vectors outside the kernel of the skew-symmetrized Euler form, the DT invariants will be independent of the choice of the potential $W$ (see Corollary \ref{cor: Windependence}).

\subsection{Attractor DT invariants}
The DT invariants $\Omega_\gamma^{+,\theta}$ of a quiver with potential $W$ are locally constant functions of the $\gamma$-general stability parameter $\theta \in \gamma^{\perp}$. Their change across the loci of non-$\gamma$-general stability parameters can be computed by the wall-crossing formula of Joyce--Song \cite{JoyceSong} and Kontsevich--Soibelman \cite{kontsevich2008stability}. 
Using the wall-crossing formula, DT invariants can be expressed in terms of the simpler \emph{attractor DT invariants}, which are quiver DT invariants at specific values of the stability parameter, defined as follows. 

The skew-symmetrized Euler form $\omega_Q \colon N_Q \times N_Q \rightarrow \Z$ is defined by  
\begin{equation}
\label{Eq: Euler form}
 \omega_Q(\gamma, \gamma'):= \sum_{i,j\in Q_0}(a_{ij}-a_{ji})\gamma_i\gamma_j'\,,   
\end{equation}
where $a_{ij}$ is the number of arrows in $Q$ from the vertex $i$ to the vertex $j$.
For every $\gamma \in N_Q$, the specific point 
\[ \iota_\gamma \omega_Q :=\omega_Q ( \gamma,- ) \in \gamma^{\perp} \subset M_{\RR} \] 
is called the \emph{attractor point} for $\gamma$ \cite{AlexandrovPioline, mozgovoy2020attractor}. In general, the attractor point is not $\gamma$-general and we define
the attractor DT invariants $\Omega_\gamma^{*}$ by
\begin{equation}\nonumber
\Omega_\gamma^{+,*}
\coloneqq  \Omega_\gamma^{+,\theta_\gamma}\,,\end{equation} 
where $\theta_\gamma$
is a small $\gamma$-general perturbation of 
$\omega_Q ( \gamma,- )$ in $\gamma^{\perp}$ \cite{AlexandrovPioline, mozgovoy2020attractor}. The integer $\Omega_\gamma^{+,*}$ is independent of the choice of the small perturbation
\cite{AlexandrovPioline, mozgovoy2020attractor}. Thus, for a fixed dimension vector $\gamma \in N$, we have a well-defined attractor DT invariant $\Omega_\gamma^{+,\star}$ -- for detailed discussion of these invariants see \cite{AlexandrovPioline, KS, mozgovoy2020attractor}.
By iterative applications of the wall-crossing formula, general DT invariants are uniquely determined in terms of the attractor DT invariants. 
This reconstruction of the general DT invariants from the attractor DT invariants can be made explicit using either the flow tree formula \cite{ABflow} or the attractor tree formula \cite{mozgovoy2022operadic}.

We denote by $I \subset Q_0$ the set of vertices $i$ of $Q_0$ such that $\iota_{s_i} \omega_Q \neq 0$, that is, $s_i \notin \ker \omega_Q$. In this paper, we will mainly consider quivers with potentials having a very simple set of attractor DT invariants, as made precise in the following definition.

\begin{definition} \label{def_trivial_attractor_invariants}
    A quiver with potential $(Q,W)$ has \emph{trivial attractor DT invariants} if 
    \begin{itemize}
    \item[(i)] $\Omega_{s_i}^{+,\star} =1$ for all $i\in I$, and
    \item[(ii)] $\Omega_\gamma^{+,\star} =0$ for all $\gamma \in N_Q$ such that $\gamma\neq s_i$ for all $i \in Q_0$, and  $\gamma \notin \ker \omega_Q$.
    \end{itemize}
\end{definition}

\begin{remark} \label{rem_s_i}
    For every $k \in \Z_{\geq 1}$, the only decompositions of $ks_i$ in $N_Q$ with positive coefficients contain only multiples of $s_i$. As $\omega_Q(s_i,s_i)=0$, it follows from the wall-crossing formula that $\Omega_{ks_i}^{+,\theta}$ does not depend on $\theta$. In particular, if $\Omega_{s_i}^{+,\star} =1$ as in Definition \ref{def_trivial_attractor_invariants}(i), then we have $\Omega_{s_i}^{+,\theta} =1$ for all $\theta \in s_i^\perp$. Similarly, if $\Omega_{ks_i}^{+,\star} =0$ for all $k>1$ as in Definition \ref{def_trivial_attractor_invariants}(i), then we have $\Omega_{ks_i}^{+,\theta} =0$ for all $k>1$ and $\theta \in s_i^\perp$.
\end{remark}

\begin{remark} \label{rem_ker}
If $\gamma \in \ker \omega_Q$, then, by the wall-crossing formula, $\Omega_{\gamma}^{+,\theta}$ does not depend on $\theta$, and so in particular we have $\Omega_{\gamma}^{+,\theta}=\Omega_{\gamma}^{+,\star}$ for all $\theta \in \gamma^\perp$. It also follows from the wall-crossing formula that these invariants do not play a role in any wall-crossing of other invariants  $\Omega_{\gamma'}^{+,\theta}$. In particular, if $\gamma' \notin \ker \omega_Q$, then $\Omega_{\gamma'}^{+,\theta}$ can be recovered from the attractor DT invariants $\Omega_{\gamma}^{+,\theta}$ with $\gamma \notin \ker \omega_Q$.
\end{remark}

\begin{remark} 
A closely related notion has been introduced in \cite[Def 7.3]{davison2021positivity}: a quiver is called \emph{genteel} if $\Omega_{\gamma}^{+,\star} =0$ unless $\gamma \in \Z_{\geq 1} s_i$ for some $i \in Q_0$. However, to be genteel and to have trivial attractor DT invariants are slightly different in general. For example, we allow $\Omega_{\gamma}^{+,\star} \neq 0$ if $\gamma \in \ker \omega_Q$ in Definition \ref{def_trivial_attractor_invariants}. On the other hand, for a genteel quiver as in \cite[Def 7.3]{davison2021positivity}, to have $\Omega_{k s_i}^{+,\star} \neq 0$ for $k>1$ is allowed, whereas it is not for a quiver with trivial DT invariants as in Definition \ref{def_trivial_attractor_invariants}.
\end{remark}

DT invariants of a quiver with potential having trivial attractor DT invariants have particularly nice positivity properties, as illustrated by the following result.

\begin{theorem} \label{thm_nice_dt}
Let $(Q,W)$ be a quiver with potential having trivial attractor DT invariants. Then, for every $\gamma \in N_Q \setminus \ker \omega_Q$ and every $\gamma$-general stability parameter $\theta \in \gamma^\perp$, the DT invariant $\Omega_\gamma^{+,\theta}$ is a non-negative integer. 
\end{theorem}

\begin{proof}
By Remark \ref{rem_ker}, DT invariants $\Omega_\gamma^{+,\theta}$ with $\gamma \notin \ker \omega_Q$ can be reconstructed using the wall-crossing formula from attractor DT invariants $\Omega_{\gamma'}^{+,\star}$ with $\gamma' \notin \ker \omega_Q$.
For a quiver with potential having trivial attractor DT invariants, the attractor DT invariants $\Omega_{\gamma'}^{+,\star}$ with $\gamma' \notin \ker \omega_Q$ are either $0$ or $1$, and so in particular are positive. The result follows because positivity of DT invariants is preserved under wall-crossing by the proof of \cite[Thm 1.13]{GHKK} (see also \cite{davison2021positivity}).
\end{proof}

In many situations the attractor DT invariants are known to be trivial. We provide some examples below.

\begin{example} \label{ex_1}
It is shown by Bridgeland \cite{Bridgeland} (see also \cite[Lemma 7.5]{davison2021positivity}) that if $Q$ is acyclic then 
\begin{equation}
\label{eq: Bridgeland-Mou}
   \Omega_\gamma^{+,\star}= \begin{cases} 
      1 & \mathrm{if} \,\ \gamma = s_i \,\,\text{for some}\,\, i\in Q_0 \\
      0 & \mathrm{otherwise}
   \end{cases} 
\end{equation}
and so in particular $Q$ has trivial attractor  DT invariants.
In \cite{mou2021scattering}, Lang Mou shows
more generally that the attractor DT invariants of a 2-acyclic quiver with non-degenerate potential $(Q,W)$ which admits a so called green-to-red sequence are also given as in \eqref{eq: Bridgeland-Mou}. Many quivers of interest in representation theory admit green-to-red sequences -- see for example \cite{shen2021cluster, WengDTBruhat}.
It is also known that among the finite mutation quivers, which include the quivers associated with triangulations of surfaces \cite{LF}, all of them admit a green-to-red sequence, except those associated with one-punctured surfaces of genus $g \geq 1$ and the so called $X_7$ quiver
\cite{mills2017maximal}. Moreover, it follows from \cite{chen2023stability} that \eqref{eq: Bridgeland-Mou} still holds for the quivers associated with the one-punctured surfaces of genus $g \geq 2$, despite the fact that they do not admit green-to-red sequences. Finally, by \cite[Cor 1.2 (ii)]{mou2021scattering}, \eqref{eq: Bridgeland-Mou} does not hold for the quiver $Q$ associated to the once-punctured torus, but still $Q$ has trivial attractor DT invariants because the dimension vector of the additional non-zero attractor DT invariant is contained in $\ker \omega_Q$.
\end{example}

\begin{example} \label{ex_2}

Given a toric Calabi-Yau 3-fold $X$, one can construct a quiver with a potential $(Q,W)$ such that
\[D^bRep(Q,W) \cong D^bCoh(X),\]
where $D^bRep(Q,W)$ is the bounded derived category of representations of $(Q,W)$ and $D^bCoh(X)$ is the bounded derived category of coherent sheaves on $X$ \cite{mozgovoy2009crepant}. 
When $X$ admits compact divisors,
Beaujard--Manschot--Pioline \cite{beaujard2020vafa} and Mozgovoy--Pioline \cite{mozgovoy2020attractor} conjecture that 
$(Q,W)$ admits trivial attractor DT invariants.
Moreover, Descombes formulates in \cite[Conj 1.3]{MR4524190} a conjecture for the values of the non-zero attractor DT invariants $\Omega_\gamma^\star$ with $\gamma \in \ker \omega_Q$. Both conjectures are proved for $X$ equal to the local projective plane by Bousseau--Descombes--Le Floch--Pioline \cite[Thm 1]{bousseau2022bps}.
\end{example}

We provide below an explicit example of quiver with potential having non-trivial attractor DT invariants, following \cite[\S 5.2.3]{denef2011split}. 
\begin{example}
Let $a$, $b$, and $c$ be three distinct positive integers such that $a+b \geq c$,
$b+c \geq a$ and $c+a \geq b$. Let $Q$ be the 3-gon quiver with three vertices $s_1$, $s_2$, $s_3$, and $a$ arrows from $s_3$ to $s_2$, $b$ arrows from $s_2$ to $s_1$, and $c$ arrows from $s_1$ to $s_3$. For example, one can take $a=5$, $b=4$, and $c=3$, as in Figure \ref{Fig:badquiver}.
Fix the dimension vector $\gamma=(1,1,1)$. The corresponding attractor point is given by $\iota_\gamma \omega_Q = (c-a, a-b, b-c)$, and so we have in particular that $\gamma \notin \ker \omega_Q$. On the other hand, it is shown in \cite[\S 5.2.3]{denef2011split}, that for $W$ a generic cubic potential for $Q$, if the stability parameter $\theta=(\theta_1,\theta_2,\theta_3)$ satisfies $\theta_1 <0$ and $\theta_3 >0$, then, 
the critical locus of the trace function $\mathrm{Tr} (W)$ on the moduli space $\mathcal{M}_\gamma^\star$ for a stability parameter close to the attractor point is a smooth complete intersection $X$ of $c$ hypersurfaces of bi-degree $(1,1)$ in $\PP^{a-1} \times \PP^{b-1}$. Moreover, $\mathrm{Tr} (W)$ is transverse in the directions normal to its critical locus, and so $\Omega_{\gamma}^{+,\star}=e(X)$, which is non-zero in general (for example if $\dim X=a+b-2-c$ is even, because the cohomology of $X$ is then concentrated in even degrees by the Lefschetz hyperplane theorems). As $\Omega_{\gamma}^{+,\star} \neq 0$ and $\gamma \notin \ker \omega_Q$, it follows that $(Q,W)$ does not have trivial attractor DT invariants. Additional explicit formulas and asymptotics for the Euler characteristic $e(X)$ can be found in \cite[App E]{denef2011split} and \cite[\S 4]{lee2012quiver}.
\end{example}

\begin{figure}[h]
\center{\includegraphics{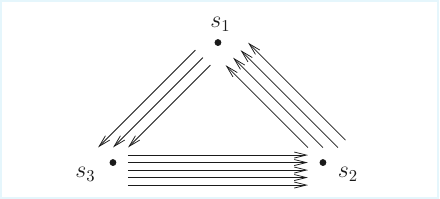}}
\caption{A 3-node quiver with non-trivial attractor DT invariants for a generic cubic potential.}
\label{Fig:badquiver}
\end{figure}

\subsection{Stability scattering diagram}

Given a quiver $Q$, we denote
\[ N_Q^\oplus:=\{ \gamma =\sum_{i \in I} \gamma_i s_i \in N_Q \,|\, \gamma_i \in \Z_{\geq 0} \} \,,\]
and \[ N_Q^+ := N_Q^\oplus \setminus \{0\}\,.\]
We denote by $\Q[N_Q^{\oplus}]$ the monoid $\Q$-algebra of $N_Q^{\oplus}$, that is the algebra of polynomials 
\[ \sum_{\gamma \in N_Q^{\oplus}} c_\gamma z^\gamma \] 
with $c_\gamma \in \Q$. Let $\mathfrak{m}$ be the maximal ideal of $\Q[N_Q^{\oplus}]$, generated by the monomials $z^\gamma$ with $\gamma \in N_Q^+$. We denote by $\Q[\![N_Q^{\oplus}]\!]$ the completion of $\Q[N_Q^{\oplus}]$ with respect to this maximal ideal: it is the algebra of formal power series 
\[ \sum_{\gamma \in N_Q^{\oplus}} c_\gamma z^\gamma \] with for every $k \in \Z_{\geq 0}$ finitely many coefficients $c_\gamma \neq 0$ with $\gamma=\sum_{i\in I}\gamma_i s_i$ and $\sum_{i \in I}\gamma_i \leq k$. 

\begin{definition} \label{def_dt_wall}
Let $Q$ be a quiver with $d$ vertices. 
A \emph{$Q$-wall} is a pair $(\fod, f_\fod)$, where 
\begin{itemize}
    \item[(i)] $\fod$ is a $(d-1)$-dimensional convex rational polyhedral cone in $M_{Q,\RR}$ contained in an hyperplane of the form $\gamma_\fod^\perp$, with $\gamma_\fod \in N_Q^+$ primitive.
    \item[(ii)] $f_\fod \in \Q[\![z^{\gamma_\fod}]\!] \subset \Q[\![N_Q^{\oplus}]\!]$ is a formal power series in $z^{\gamma_\fod}$ with constant term $1$, that is of the form \begin{equation} \label{eq_Q_wall}
    f_\fod = 1+\sum_{\gamma \in \Z_{\geq 1} \gamma_\fod} c_\gamma z^\gamma\end{equation}
    with $c_\gamma \in \Q$.
\end{itemize}
\end{definition}

\begin{remark}
    Note that $\gamma_\fod$ in Definition \ref{def_dt_wall}(i) is uniquely determined by $\fod$.
\end{remark}

\begin{definition} \label{def_incoming_Q_wall}
A $Q$-wall $(\fod, f_\fod)$ is \emph{incoming}
if $\iota_{\gamma_\fod} \omega_Q \in \fod$, that is, if $\fod$ contains the attractor point for $\gamma_\fod$.
\end{definition}

\begin{definition} \label{def_dt_scattering} 
A \emph{$Q$-scattering diagram} is a set $\foD=\{(\fod, f_\fod)\}$ of $Q$-walls such that for every $k \in \Z_{\geq 0}$, there exist finitely many walls $(\fod, f_\fod) \in \foD$ with $f_\fod \neq 1 \mod \mathfrak{m}^k$. 
\end{definition}

We define the \emph{support} of
a $Q$-scattering diagram $\foD=\{(\fod, f_\fod) \}$ in $M_{Q,\RR}$ by
\begin{equation}\nonumber
\mathrm{Supp}(\foD)=\bigcup_{\fod \in \foD} \fod \subset M_{Q,\RR} \,, 
\end{equation}
and its
\emph{singular locus} by 
\begin{equation} \nonumber
\mathrm{Sing}(\foD)=\bigcup_{\fod \in \foD} \partial \fod \cup \bigcup_{\fod, \fod' \in \foD} \fod \cap \fod' \subset M_{Q,\RR} \,,\end{equation}
where $\partial \fod$ denotes the boundary of a wall $\fod$, and the last union is over the pairs of walls $\fod, \fod'$ such that $\mathrm{codim} \fod \cap \fod' \geq 2$.
For every point $x\in M_{Q,\RR} \setminus \Sing(\foD)$, we define
\begin{equation} \nonumber
f_{\foD,x}:= \prod_{\substack{\fod \in \foD\\ x \in \fod}} f_{\fod} \,,
\end{equation}
where the product is over all the walls of $\foD$ containing $x$.
Two $Q$-scattering diagrams $\foD$ and $\foD'$ are called \emph{equivalent} if 
\begin{equation} \nonumber
f_{\foD,x}=f_{\foD', x}
\end{equation}
for all $x \in M_{Q,\RR} \setminus (\mathrm{Sing}(\foD) \cup \mathrm{Sing}(\foD'))$.

Given a path
 \begin{align*} \alpha: [0,1] &\longrightarrow M_{Q,\RR} \setminus \mathrm{Sing}(\foD)\\ 
 \tau &\longmapsto \alpha(\tau)\,\end{align*}
intersecting transversally $\Supp(\foD)$, and $\tau \in [0,1]$ such that $\alpha(\tau)\in \Supp(\foD)$, we define an automorphism of $\Q$-algebras 
\begin{align} \label{eq_autom_Q}
\mathfrak{p}_{\foD, \alpha, \tau} : \Q[\![N_Q^{\oplus}]\!] &\longrightarrow \Q[\![N_Q^{\oplus}]\!]\\
z^{\gamma} &\longmapsto f_{\foD,\alpha(\tau)}^{\epsilon_{\alpha,\tau}\,\omega_Q(\gamma_{\alpha,\tau} ,\gamma)} z^{\gamma}\,,
\nonumber
\end{align}
where $\gamma_{\alpha,\tau}$ is the unique primitive element in $N_Q^+$ such that $\fod \subset \gamma_{\alpha,\tau}^\perp$ for all walls $\fod$ containing $\alpha(\tau)$, and $\epsilon_{\alpha,\tau} \in \{\pm 1\}$ is the sign such that $\epsilon_{\alpha,\tau}\, \omega_Q (\gamma_{\alpha,\tau}, \alpha'(\tau)) <0$.
The path ordered automorphism 
$\mathfrak{p}_{\foD, \alpha}$ is the product of automorphisms
$\mathfrak{p}_{\foD, \alpha, \tau}$ for all $\tau \in [0,1]$ such that $\alpha(\tau)\in \Supp(\foD)$, and ordered following increasing values of $\tau$.
Finally, a scattering diagram is called \emph{consistent} 
if for any path $\alpha$ in $M_\RR \setminus \mathrm{Sing}(\foD)$ with $\alpha(0)=\alpha(1)$ 
the associated path ordered automorphism 
$\mathfrak{p}_{\foD, \alpha}$ is the identity.

\begin{definition} \label{def_central}
A $Q$-wall $\fod$ is called \emph{central} if $\gamma_\fod \in \ker \omega_Q$. 
\end{definition}

Given a $Q$-scattering diagram $\foD$, there is a natural way to produce a $Q$-scattering diagram $\overline{\foD}$ without central walls: $\overline{\foD}$ is simply the set of non-central walls of $\foD$.

\begin{lemma} \label{lem_no_central}
If a $Q$-scattering diagram $\foD_Q$ is consistent, then the corresponding $Q$-scattering diagram $\overline{\foD}_Q$ without central walls is also consistent.
\end{lemma}

\begin{proof}
It follows from \eqref{eq_autom_Q} that the automorphism $\mathfrak{p}_{\foD,
\alpha,\tau}$ is the identity if $\gamma_{\alpha,\tau} \in \ker \omega_Q$. Hence, removing the central walls does not change the automorphisms $\mathfrak{p}_{\foD,\alpha}$: we have $\mathfrak{p}_{\overline{\foD},\alpha}=\mathfrak{p}_{\foD,\alpha}$ for all paths $\alpha$. In particular, this implies that, if $\foD$ is consistent, then $\overline{\foD}$ is also consistent.
\end{proof}

We now review following \cite{Bridgeland} how DT invariants of a quiver with potential can be organized into a scattering diagram called the stability stability scattering.

\begin{definition}
Let $(Q,W)$ be a quiver with potential. The \emph{stability scattering diagram} $\foD^{\mathrm{st}}_{(Q,W)}$ is the unique $Q$-scattering diagram up to equivalence  such that for every primitive dimension vector $\gamma_0 \in N_Q^{\oplus}$, and general point $\theta \in \gamma_0^{\perp}$, we have 
\begin{equation} \label{eq_stab_scattering}
f_{\foD_{(Q,W)}^\mathrm{st},\theta}
= \exp \left( \sum_{\gamma \in \Z_{\geq 1}\gamma_0} |\gamma| \,\,\overline{\Omega}_\gamma^{+,\theta} z^\gamma \right) \,.
\end{equation}
\end{definition}

According to \cite[Thm 1.1]{Bridgeland}, the stability scattering diagram $\foD^{\mathrm{st}}_{(Q,W)}$ is consistent. 
It follows from Lemma \ref{lem_no_central}
that the corresponding scattering diagram without central walls $\overline{\foD}^{\mathrm{st}}_{(Q,W)}$
is also consistent.

Following \cite[\S 11.4]{Bridgeland}, we review the notion of a cluster scattering diagram.

\begin{definition}
Let $Q$ be a quiver. The \emph{initial cluster scattering diagram} $\foD^{\mathrm{cl}}_{Q, \mathrm{in}}$ is the $Q$-scattering diagram with set of walls 
\begin{equation} \label{eq_initial_cluster}
\foD^{\mathrm{cl}}_{Q, \mathrm{in}} :=
\{ (s_i^{\perp}, 1+z^{s_i})\}_{i \in I}\,.
\end{equation}
\end{definition}

Note that all walls of $\foD^{\mathrm{cl}}_{Q, \mathrm{in}}$ are incoming as $\iota_{s_i} \omega_Q \in s_i^{\perp}$ for all $i\in I$.

By \cite[Prop 3.3.2]{KS} (see also \cite[Thm 1.21]{GHKK}), there exists a unique up to equivalence $Q$-scattering diagram $\foD^{\mathrm{cl}}_{Q}$ containing the initial cluster scattering diagram $\foD^{\mathrm{cl}}_{Q, \mathrm{in}}$ and such that every wall $\fod \in\foD^{\mathrm{cl}}_{Q} \setminus \foD^{\mathrm{cl}}_{Q, \mathrm{in}}$ is non-incoming.  We refer to $\foD^{\mathrm{cl}}_{Q}$ as the \emph{cluster scattering diagram}.

\begin{theorem} \label{thm_dt_cluster}
Let $(Q,W)$ be a quiver with potential having trivial attractor DT invariants as in Definition \ref{def_trivial_attractor_invariants}. 
Then, the scattering diagram $\overline{\foD}^{\mathrm{st}}_{(Q,W)}$ without central walls corresponding to
the stability scattering diagram $\foD^{\mathrm{st}}_{(Q,W)}$ and the cluster scattering diagram $\foD^{\mathrm{cl}}_{Q}$ are equivalent.
\end{theorem}

\begin{proof}
By Definition \ref{def_trivial_attractor_invariants} of trivial attractor DT invariants, we have $\Omega_{s_i}^{+,\star}=1$ and $\Omega_{ks_i}^{+,\star}=0$ for all $i \in I$ and $k \in \Z_{>1}$. 
By Remark \ref{rem_s_i}, it follows that, for all $\theta \in s_i^\perp$, $i \in I$, and $k \in \Z_{>1}$, we have $\Omega_{s_i,\theta}=1$ and $\Omega_{ks_i}^{+,\theta}=0$.
Applying \eqref{eq_dt_rational}, we obtain 
$\overline{\Omega}_{ks_i}^\theta=\frac{(-1)^{k-1}}{k^2}$ for all $k \in \Z_{\geq 1}$, and so, by \eqref{eq_stab_scattering}, we have
\[ f_{\foD_{(Q,W)}^\mathrm{st},\theta}
= \exp \left( \sum_{k \in \Z_{\geq 1}} k \,\,\frac{(-1)^{k-1}}{k^2}
z^{k s_i} \right) = 1+z^{s_i} \,,\]
for all $\theta \in s_i^\perp$. Moreover, for $i \in I$, we have $s_i \notin \ker \omega_Q$, and so we have $f_{\foD_{(Q,W)}^\mathrm{st},\theta}=f_{\overline{\foD}_{(Q,W)}^\mathrm{st},\theta}$ for all $\theta \in s_i^\perp$. Hence, there is a representative $\overline{\foD}$ of the equivalence class of $\overline{\foD}_{(Q,W)}^\mathrm{st}$ containing $(s_i^\perp, 1+z^{s_i})$ as incoming walls for all $i\in I$, that is, such that $\foD^{\mathrm{cl}}_{Q, \mathrm{in}} \subset \overline{\foD}$
by \eqref{eq_initial_cluster}.

By Definition \ref{def_trivial_attractor_invariants}, we also have $\overline{\Omega}_{\gamma}^{+,\star}=0$ for all $\gamma \in N_Q^+$ such that either $\gamma \neq s_i$ for some $i\in I$ or $\gamma \notin \ker \omega_Q$. It follows that $\overline{\foD}$ does not have any other incoming walls apart from $(s_i^\perp, 1+z^{s_i})$ for $i \in I$, that is, all walls in $\overline{\foD} \setminus \foD^{\mathrm{cl}}_{Q, \mathrm{in}}$ are non-incoming. Therefore, $\overline{\foD}$ is equivalent to the cluster scattering diagram $\foD^{\mathrm{cl}}_{Q}$ by the uniqueness part of \cite[Prop 3.3.2]{KS} (see also \cite[Thm 1.21]{GHKK}).
\end{proof}

\begin{corollary}
\label{cor: Windependence}
Let $Q$ be a quiver and $W$ and $W'$ be two potentials such that $(Q,W)$ and $(Q,W')$ have trivial attractor DT invariants. Then, DT invariants of $(Q,W)$ and $(Q,W')$ are the same for dimension vectors not in $\mathrm{Ker}\, \omega_Q$.
\end{corollary}

\begin{proof}
This is a direct consequence of Theorem \ref{thm_dt_cluster}, since both for $(Q,W)$ and $(Q,W')$ the stability scattering diagram, encoding the data of quiver DT invariants with dimension vectors not in $\mathrm{Ker}\, \omega_Q$, is equivalent to the cluster scattering diagram which does not depend on the choice of the potential $W$ or $W'$. 
\end{proof}

\section{Punctured Gromov--Witten invariants of cluster varieties}

\subsection{Cluster varieties}
\label{sec:cluster}

\begin{definition} \label{def_seed}
A \emph{symplectic seed} $\mathbf{s}=(N, (e_i)_{i \in I}, \omega)$ is the data of a finite rank free abelian group $N$, a collection of elements $e_i \in N$ indexed by a finite set $I$, and an integral skew-symmetric form 
\[ \omega \colon N \times N \rightarrow \Z\,,\]
such that $\ker \omega=0$, that is, such that $\omega \otimes \Q$ is non-degenerate.
\end{definition}

Given a symplectic seed $\mathbf{s}=(N, (e_i)_{i \in I}, \omega)$, we denote $M:=\Hom(N,\Z)$ and $M_\RR:=M \otimes \RR = \Hom (N,\RR)$, so that  $\omega \in \bigwedge^2 M$. As $\ker \omega =0$, it follows that the map 
\begin{align}\nu: N &\longrightarrow M \\
\gamma &\longmapsto \iota_\gamma \omega := \omega(\gamma, -) \nonumber
\end{align}
is injective and has finite cokernel. In particular, the image $\mathrm{Im}(\nu)$ is of finite index in $M$.
For every $i \in I$, we denote 
\[ v_i := \nu(e_i)=\iota_{e_i}
\omega = \omega(e_i,-) \in M\,.\]
In addition, we will make the following assumption:
\begin{equation}\label{eq_cluster_assumption}
 |v_i|=1 \,\,\text{for all}\,\, i\in I\,,
\end{equation}
where $|v_i|$ is the divisibility of $v_i$ in $M$.

We review the construction of a cluster variety starting from a symplectic seed $\mathbf{s}=(N, (e_i)_{i \in I}, \omega)$ satisfying \eqref{eq_cluster_assumption}. Let $\Sigma$ be a fan in $M_\RR$ containing the rays $\RR_{\geq 0}v_i$ for all $i\in I$, and such that the corresponding toric variety $X_\Sigma$ is smooth and projective. Such a fan always exists by toric resolution of singularities \cite[Thm 11.1.9]{toric_book} and the toric Chow lemma 
\cite[Thm 6.1.18]{toric_book}.
For every $i\in I$, we denote by $\overline{D}_i$ the toric divisor of $X_\Sigma$ corresponding to the ray $\RR_{\geq 0} v_i$ of $\Sigma$. 
For every $i \in I$, fix $t_i \in \kk^{\star}$, and define the hypersurface $H_i \subset \overline{D}_i$ as the closure in $\overline{D}_i$ of the locus of equation 
\begin{equation} \label{eq_H_i}
1+t_i z^{e_i}=0\,.\end{equation}
Up to refining $\Sigma$, we can assume that $H_i$ is smooth for every $i\in I$.
Moreover, 
we assume that for every $i,j \in I$ , we have 
\begin{equation} \label{eq_H_int}
H_i \cap H_j =\emptyset \,.
\end{equation}
Note that if $\RR_{\geq 0}v_i \neq \RR_{\geq 0}v_j$
for all $i,j \in I$, then up to refining $\Sigma$, one can assume that $\overline{D}_i \cap \overline{D}_j=\emptyset$ for all $i,j \in I$, and so that \eqref{eq_H_int} is satisfied. If $\dim X=2$, and $t_i \neq t_j$ for all $i, j \in I$, then 
\eqref{eq_H_int} is also satisfied.
Both assumptions \eqref{eq_cluster_assumption} and
\eqref{eq_H_int} will be necessary to apply the main result of \cite{HDTV}. We expect that these technical assumptions can be dropped after an appropriate generalization of \cite{HDTV}.

Let $X$ be the blow-up of $X_\Sigma$ along the codimension two locus 
\[ H :=\bigcup_{i\in I} H_i\,,\]
and let $D \subset X$ be the strict transform of the toric boundary divisor  $D_\Sigma$ of $X_\Sigma$. As $H$ is smooth, it follows that $X$ is also smooth and $D$ is a simple normal crossing divisor in $X$.

For every $i \in I$, we denote by $D_i$ the irreducible component of $D$ obtained as strict transform of the toric divisor $\overline{D}_i$ of $X_\Sigma$.
We refer to the complement $U=X \setminus D$ as the \emph{cluster variety} defined by the symplectic seed $\mathbf{s}$, and $(X,D)$ as a \emph{log Calabi-Yau compactification} of the cluster variety. As we will always consider the pair $(X,D)$ up to locally trivial deformations, we suppress the parameters $t_i$ from the notation and from the terminology. The skew-symmetric form $\omega$ on $N$ naturally defines a Poisson structure on $U$, which is in fact, by the assumption $\ker \omega=0$, non-degenerate and defines a holomorphic symplectic form on $U$ with first order poles along $D$.

\begin{remark}  \label{remark_finite_group}
In the usual theory of cluster varieties \cite{FG2, GHKbirational}, 
a (skew-symmetric) \emph{seed} $\tilde{\mathbf{s}} =(\tilde{N}, (\tilde{e}_i)_{i\in I}, \tilde{\omega})$
consists of a finite abelian group $\tilde{N}$, a basis $(\tilde{e}_i)_{i\in I}$ of $\tilde{N}$, and a skew-symmetric form $\tilde{\omega}$.
Given a symplectic seed $\mathbf{s}=(N, (e_i)_{i \in I}, \omega)$ as in Definition \ref{def_seed}, one can define a seed $\tilde{\mathbf{s}} =(\tilde{N}, (\tilde{e}_i)_{i\in I}, \tilde{\omega})$ by $\tilde{N}:=\bigoplus_{i\in I}\Z \tilde{e}_i$, and $\tilde{\omega}=\pi^\star \omega$, where $\pi$ is the projection 
\begin{align*}
\pi: \tilde{N} \longrightarrow N \\ 
\tilde{e}_i \longmapsto e_i \,.
\end{align*}
If $(e_i)_{i \in I}$ generate $N$, then $N$ is exactly the quotient of $\tilde{N}$ by the kernel $\ker \tilde{\omega}$. In this case, $U$ is a symplectic fiber of the Poisson $\mathcal{X}$ cluster variety defined by the seed $\tilde{\mathbf{s}}$ \cite{FG2, GHKbirational}.
In general, $N$ contains $\tilde{N}/
\ker \tilde{\omega}$ as a sublattice of finite index, and $U$ is a quotient by the finite group $N/\pi(\tilde{N})$ of a symplectic fiber of the Poisson $\mathcal{X}$ cluster variety defined by $\tilde{\mathbf{s}}$.

\end{remark}

\subsection{Punctured Gromov--Witten invariants}
\label{sec_gw}

Given a log Calabi-Yau pair, Gross--Siebert define in \cite{gross2021canonical}
punctured Gromov--Witten invariants indexed by wall types. Using these invariants, they define the canonical scattering diagram which can be used to construct the mirror geometry.
In this section, we briefly review the definition of these punctured Gromov--Witten invariants in the particular setting of log Calabi-Yau compactifications of cluster varieties satisfying assumptions \eqref{eq_cluster_assumption}-\eqref{eq_H_int} of \S\ref{sec:cluster}.

Let $(X,D)$ be a $d$-dimensional log Calabi-Yau compactification of a cluster variety satisfying assumptions \eqref{eq_cluster_assumption}-\eqref{eq_H_int} of \S\ref{sec:cluster}.
Let $(B, \mathscr{P})$ be the tropicalization of $(X,D)$ as in \cite[\S 2.1.1]{HDTV}: $\mathscr{P}$ is a collection of cones, containing a cone $\RR_{\geq 0}^r$ for each codimension $r$ stratum of $(X,D)$, and $B$ is the topological space obtained by gluing together the cones of $\scrP$ according to the incidence relations between strata. By construction, one can view $\mathscr{P}$ as a decomposition of $B$ into a union of cones. 
Let $\Delta \subset B$ be the union of codimension $2$ cones of $\mathscr{P}$. Then, as reviewed in \cite[(2.4)]{HDTV}, there is a natural integral affine structure
on the complement $B_0:=B 
\setminus \Delta$ of $\Delta$ in $B$.
As described in \cite[Thm 3.4]{HDTV}, this integral affine structure actually extends over the complement of a smaller discriminant locus, but we will not use this fact.
For every cone $\sigma \in \mathscr{P}$, we denote by $\Lambda_\sigma$ the space of integral tangent vectors to $\sigma$, and for every point $x \in B_0$, we denote by $\Lambda_x$ the space of integral tangent vectors to $B_0$ at $x$.

Following \cite{ACGS, ACGSI, gross2021canonical}, we review below the necessary tropical language to define moduli spaces of punctured maps.

\begin{definition}
A \emph{tropical type} is the data of a triple $\tau=(G, \boldsymbol{\sigma}, \mathbf{u})$, where:
\begin{itemize}
\item[(i)] $G$ is a graph, with set of vertices $V(G)$, set of edges $E(G)$, and set of legs $L(G)$, 
\item[(ii)] $\boldsymbol{\sigma}$ is a map 
\[ \boldsymbol{\sigma}: V(G) \cup E(G) \cup L(G) \longrightarrow \mathscr{P}\,,\]
such that $\boldsymbol{\sigma}(v) \subset \boldsymbol{\sigma}(E)$ if $v$ is a vertex of an edge or leg $E$, 
\item[(iii)] $\mathbf{u}$ is a map 
\[ \mathbf{u}: \{ (V,E)\,|\, V \in V(G) \,, E \in E(G) \cup L(G)\,, V \in \partial E \} \longrightarrow \bigcup_{\sigma \in \mathscr{P}} \Lambda_\sigma \,,\]
such that $\mathbf{u}(V,E) \in \Lambda_{\boldsymbol{\sigma}(E)}$ for all $(V,E)$ such that $V \in \partial E$, and $\mathbf{u}(V_2,E)=-\mathbf{u}(V_1,E)$ if $E \in E(G)$ and $\partial E=\{V_1, V_2\}$.
\end{itemize}
\end{definition}

\begin{definition}
An \emph{abstract tropical curve} is a pair $(G,\ell)$
consisting of a graph $G$ and of length functions $\ell: E(G) \cup L(G) \rightarrow \RR_{\geq 0} \cup \{\infty\}$ such that $\ell(E) \neq \infty$ for all $E\in E(G)$.
\end{definition}

We view an abstract tropical curve $(G,\ell)$ as a metric graph, with edges $E \in E(G)$ identified with the line segment $[0,\ell(E)]$, and with legs $L \in L(G)$ identified with the line segment $[0,\ell(L)]$ if $\ell(L)\neq \infty$, and with the half-line $[0,+\infty)$
if $\ell(L)=\infty$.

\begin{definition} \label{def_tropical_map}
Let $\tau=(G,\boldsymbol{\sigma}, \mathbf{u})$ be a tropical type. A \emph{tropical map} to $(B,\mathscr{P})$ of type $\tau$ is the data of an abstract tropical curve $(G,\ell)$ and of a map $h \colon G \rightarrow B$ such that
\begin{itemize}
\item[(i)] $h(v) \in \Int(\boldsymbol{\sigma}(v))$ for all $v \in V(G)$
\item[(ii)] for all $E \in E(G) \cup L(G)$, $h(\Int(E)) \in \Int(\boldsymbol{\sigma}(E))$ 
and $h|_E: E \rightarrow \boldsymbol{\sigma}(E)$ is an affine linear map.
\item[(iii)] for all $E \in E(G)$ with $\partial E=\{V_1, V_2\}$, we have $h(V_2)-h(V_1)=\ell(E) \mathbf{u}(V_1,E)$.
\item[(iii)] for all $L \in L(G)$ with $\partial L=\{V\}$, we have $h(L)=V+[0,\ell(L)]u(V,L)=\boldsymbol{\sigma}(L) \cap (V+\RR_{\geq 0} u(V,L))$ if $\ell(L)\neq \infty$, and $h(L)=V+\RR_{\geq 0} u(V,L)$ if $\ell(L) =\infty$.
\end{itemize}
\end{definition}

\begin{definition}
A tropical type $\tau$ is \emph{realizable} if there exists a tropical map to $(B,\mathscr{P})$ of type $\tau$.
\end{definition}

\begin{definition} \label{def_basic_monoid}
Given a realizable tropical type $\tau=(G,\boldsymbol{\sigma}, \mathbf{u})$, the \emph{basic monoid} of $\tau$ is the monoid $Q_\tau:=\Hom(Q_\tau^\vee, \NN)$, where
\[ Q_\tau^\vee := \left\{ ((p_V)_V, (\ell_E)_E) \in 
\prod_{V\in V(G)} \boldsymbol{\sigma}(v)_\Z 
\times \prod_{E\in E(G)}\NN \,\, \Bigg | \, \,
p_{V_2}-p_{V_1}=\ell_E \mathbf{u}(V_1,E)\, ,\forall E \in E(\G) \right\}
\] 
where $\boldsymbol{\sigma}(v)_\Z$ is the set of integral points of the cone $\boldsymbol{\sigma}(v)$  and $\partial E=\{V_1, V_2\}$.
\end{definition}\

By construction, the cone $Q_{\tau,\RR}^\vee :=\Hom(Q_\tau , \RR_{\geq 0})$ parametrizes  the tropical maps to $(B, \mathscr{P})$ of type $\tau$: for every $s =((p_V)_V, (\ell_E))\in Q_{\tau,\RR}^\vee $, the corresponding tropical map is defined by $\ell_s(E):=\ell_E$ and $h_s(V)=p_V$.

\begin{definition} \label{def_wall_type}
A \emph{wall type} is a realizable tropical type $\tau=(G, \boldsymbol{\sigma}, \mathbf{u})$ such that: 
\begin{itemize}
    \item[(i)] the graph $G$ is connected of genus zero, with a single leg: $L(G)=\{L_{\mathrm{out}}\}$.
    \item[(ii)] $\dim Q_{\tau,\RR}^\vee=d-2$ and $\dim \cup_{s \in Q_{\tau,\RR}^\vee} h_s(L_{\mathrm{out}})=d-1$, where $h_s: G \rightarrow B$ are the universal tropical maps indexed by $s \in Q_{\tau,\RR}^\vee$.
    \item[(iii)] $\tau$ is balanced: for every vertex $V \in V(G)$ such that $\dim \boldsymbol{\sigma}(V)=d$ or $d-1$, with incident edges or legs $E_1,\dots, E_m$, then, for every point $x \in \Int(\boldsymbol{\sigma}(V))$, one can view $u(V,E_1),\dots, u(V,E_m)$ as elements of $\Lambda_x$, and we have $\sum_{i=1}^m u(V,E_i)=0$ in $\Lambda_x$.
\end{itemize}
\end{definition}

Given a wall type $\tau$, we denote 
\begin{equation}\nonumber
u_\tau :=
u(V_{\mathrm{out}}, L_{\mathrm{out}}) \in \Lambda_{\boldsymbol{\sigma}(L_{\mathrm{out}})}
\,,
\end{equation}
where $V_{\mathrm{out}}$ is the vertex adjacent to the leg $L_{\mathrm{out}}$.
It follows from Definition \ref{def_wall_type}(ii) that $u_\tau \neq 0$.

We refer to \cite{ACGS} for the details of the general theory of punctured maps. A punctured map to $(X,D)$ is a diagram 
\[\begin{tikzcd}
C^{\circ}
\arrow[r, "f"]
\arrow[d]
&
X\\
W& 
\end{tikzcd}\]
in the category of log schemes, where $X$ is equipped with the divisorial log structure defined by $D$, the base $W$ is a log point, and $C^{\circ}$ is a puncturing of log curve $C \rightarrow W$ as in \cite[\S 2.1]{ACGS}. In particular, the underlying curve $\underline{C}$ obtained by forgetting the log structure is nodal, and a punctured map $f: C/W \rightarrow X$ is called stable if the underlying morphism of schemes $\underline{f}:\underline{C} \rightarrow X$ is a stable map. The tropicalization of a punctured map
$f: C/W \rightarrow X$ gives a diagram 
\[\begin{tikzcd}
\Sigma(C)
\arrow[r, "\Sigma(f)"]
\arrow[d]
&
\Sigma(X)=(B,\mathscr{P})\\
\Sigma(W)& 
\end{tikzcd}\]
which is a family of tropical maps to $(B,\scrP)$ as in Definition \ref{def_tropical_map}, whose generic tropical type $\tau=(G,\mathbf{\sigma},\mathbf{u})$ is called the type of $f: C/W \rightarrow X$. In particular, the graph $G$ is the dual graph of $C$, with vertices (resp.\ edges, legs) corresponding to the irreducible components (resp.\ nodes, marked points of $C$), the map $\mathbf{\sigma}$
specifies the strata of $(X,D)$ containing the images by $f$ of the components, nodes and marked points of $C$, and the map $\mathbf{u}$ gives the contact orders of $f$ at the nodes and marked points encoded by the log structures. Finally, the punctured map $f: C/W \rightarrow X$ is called \emph{basic} if 
$\Sigma(W)=Q_{\tau,\RR}^\vee$ and $\Sigma(C) \rightarrow \Sigma(W)$ is the universal family of tropical maps of type $\tau$, where $Q_\tau$ is the basic monoid of $\tau$ as in Definition \ref{def_basic_monoid}.

Let $N_1(X)$ be the abelian group of curve classes in $X$ modulo numerical equivalence. For each wall type $\tau$ and class $\beta \in N_1(X)$, there is a moduli space $\mathscr{M}_\tau(X,\beta)$ of basic stable punctured maps to $(X,D)$, of class $\beta$, and whose type admits a contraction morphism to $\tau$ \cite[\S 2.1]{gross2021canonical}. In particular, for such a punctured map $f:C^{\circ}/W \rightarrow X$, the log curve $C$ has a single marked point corresponding to the single leg $L_{\mathrm{out}}$ of $G$, with contact order $u_\tau$. Moreover, 
the underlying morphism of schemes is a genus zero stable map to $X$.
By \cite[Lemma 3.9]{gross2021canonical}, $\mathscr{M}_\tau(X,\beta)$ is a proper Deligne-Mumford stack and carries a natural zero-dimensional virtual fundamental class 
$[\mathscr{M}_\tau(X,\beta)]^\virt$. 
The corresponding punctured Gromov--Witten invariant is 
\begin{equation}\label{eq_gw}
N_{\tau,\beta}^{(X,D)}:=  \deg [\mathscr{M}_\tau(X,\beta)]^\virt \in \QQ\,.\end{equation}

Let $\Lambda_{\tau_{\mathrm{out}}}$ be the space of integral tangent vectors to the $(d-1)$-dimensional cone 
\begin{equation}\nonumber
\tau_{\mathrm{out}}:= \bigcup_{s \in Q_{\tau,\RR}^\vee} L_{\mathrm{out}}
\end{equation}
formed by the universal family of legs $L_{\mathrm{out}}$ over $Q_{\tau, \RR}^\vee$. The derivative of the universal tropical map $h=(h_s)_{s\in Q_{\tau,\RR}^\vee}$ induces a map 
$h_{*}: \Lambda_{\tau_{\mathrm{out}}} 
\rightarrow \Lambda_{\boldsymbol{\sigma}
(L_{\mathrm{out}})}$. The \emph{coefficient} $k_\tau$ is defined as the order of the finite torsion subgroup of the quotient $\Lambda_{\boldsymbol{\sigma}(L_{\mathrm{out}})} / h_{*}(\Lambda_{\tau_{\mathrm{out}}})$:
\begin{equation} \label{eq_coeff}
k_\tau:= | \left( \Lambda_{\boldsymbol{\sigma}(L_{\mathrm{out}})} / h_{*}(\Lambda_{\tau_{\mathrm{out}}})\right)_{\mathrm{tors}}|
\,.\end{equation}

\subsection{HDTV scattering diagram} \label{sec: hdtv}
For any log Calabi-Yau pair $(X,D)$, Gross--Siebert \cite{gross2021canonical} showed that the punctured Gromov--Witten invariants defined by \eqref{eq_gw} naturally define a canonical scattering diagram in the
tropicalization $(B,\mathscr{P})$
of $(X,D)$.
In joint work of the first author with Gross \cite{HDTV}, it was shown that, for log Calabi-Yau pairs obtained as blow-ups of toric varieties, and so in particular for the log Calabi--Yau compactifications of cluster varieties as in \S\ref{sec:cluster}, the canonical scattering diagram can be calculated from an entirely combinatorial scattering diagram in a vector space, referred to as the \emph{HDTV scattering diagram} in what follows.
In this section, we review the definition of the HDTV scattering diagram. The main result of \cite{HDTV} relating punctured Gromov--Witten invariants and the HDTV scattering diagram will be reviewed in the next section.

Let $N$ be a finite rank free abelian group and $(e_i)_{i \in I}$ a finite collection of elements $e_i \in N$ indexed by a finite set $I$. As in \S\ref{sec:cluster}, we denote $M:=\Hom(N,\Z)$ and $M_\RR :=M \otimes \RR =\Hom (N,\RR)$. 
Let $\QQ[M][\![(t_i)_{i\in I}]\!]$ be the algebra of power series in variables $t_i$ indexed by $i \in I$ and with coefficients in $\QQ[M]$. We denote by $\mathfrak{m}$ the ideal of $\QQ[M][\![(t_i)_{i\in I}]\!]$ generated by the variables $t_i$ for all $i \in I$.
We recall below the notion of scattering diagram in $M_\RR$ following \cite[\S 5]{HDTV}.

\begin{definition} \label{def_wall}
Fix a primitive $m_0 \in M \setminus \{0\}$.
A \emph{wall} in $M_{\RR}$ of direction $-m_0$ is a pair $(\fod, f_\fod)$, where 
\begin{itemize}
    \item[(i)] $\fod$ is a codimension one convex rational cone in $M_{\RR}$ contained in an hyperplane of the form $n_0^\perp$, with $n_0 \in N \setminus \{0\}$ primitive such that $m_0 \in n_0^{\perp}$.
    \item[(ii)] $f_\fod \in \Q[z^{m_0}][\![(t_i)_{i\in I}]\!] \subset \Q[M][\![(t_i)_{i\in I}]\!]$ with $f_\fod =1\! \mod \mathfrak{m}$, that is,  $f_\fod$ is a power series of the form 
    \begin{equation} \label{eq_wall}
     f_\fod =1+\sum_{\substack{m \\\in \Z_{\geq 0}m_0}} \sum_{\substack{\mathbf{A}=(a_i)_i \\ \in \NN^I \setminus \{0\}}} c_{m,\mathbf{A}} z^m \prod_{i\in I} t_i^{a_i}
    \end{equation}
    with $c_{m,\mathbf{A}}\in \QQ$, and for all $\mathbf{A}$, only finitely many $m$ such that $c_{m,\mathbf{A}}\neq 0$.
\end{itemize}
\end{definition}

\begin{definition} \label{def_incoming_wall}
A wall $(\fod, f_\fod)$ in $M_\RR$ of direction $-m_0$ is \emph{incoming} if
$m_0 \in \fod$.
\end{definition}

\begin{definition} \label{def_scattering}
A \emph{scattering diagram} in $M_\RR$ is a set $\foD=\{(\fod, f_\fod)\}$ of walls in $M_\RR$ such that for every $k \in \Z_{\geq 0}$, there exist only finitely many walls $(\fod, f_\fod) \in \foD$ with $f_\fod \neq 1 \mod \mathfrak{m}^k$. 
\end{definition}

Let $\foD=\{(\fod, f_\fod) \}$ be a scattering diagram in $M_\RR$. We define the \emph{support} of $\foD$ by
\begin{equation}\nonumber
\mathrm{Supp}(\foD)=\bigcup_{\fod \in \foD} \fod \subset M_\RR \,, 
\end{equation}
and the
\emph{singular locus} of $\foD$ by 
\begin{equation} \nonumber
\mathrm{Sing}(\foD)=\bigcup_{\fod \in \foD} \partial \fod \cup \bigcup_{\fod, \fod' \in \foD} \fod \cap \fod' \subset M_\RR \,,\end{equation}
where $\partial \fod$ denotes the boundary of a wall $\fod$, and the last union is over the pairs of walls $\fod, \fod'$ such that $\mathrm{codim} \fod \cap \fod' \geq 2$.
For every point $x\in M_\RR \setminus \Sing(\foD)$, we define
\begin{equation}
f_{\foD,x}:= \prod_{\substack{\fod \in \foD\\ x \in \fod}} f_{\fod} \,,
\end{equation}
where the product is over all the walls of $\foD$ containing $x$.
Two scattering diagrams $\foD$ and $\foD'$ are called \emph{equivalent} if 
\begin{equation} \nonumber
f_{\foD,x}=f_{\foD', x}
\end{equation}
for all $x \in M_\RR \setminus (\mathrm{Sing}(\foD) \cup \mathrm{Sing}(\foD'))$.

Given a path
 \begin{align*} \alpha: [0,1] &\longrightarrow M_{\RR} \setminus \mathrm{Sing}(\foD)\\ 
 \tau &\longmapsto \alpha(\tau)\,\end{align*}
intersecting transversally $\Supp(\foD)$, and $\tau \in [0,1]$ such that $\alpha(\tau)\in \Supp(\foD)$, we define an automorphism of $\Q[\![(t_i)_{i\in I}]\!]$-algebras 
\begin{align} \label{eq_autom}
\mathfrak{p}_{\foD, \alpha, \tau} : \Q[M][\![(t_i)_{i\in I}]\!] &\longrightarrow \Q[M][\![(t_i)_{i\in I}]\!]\\
z^m &\longmapsto f_{\foD,\alpha(\tau)}^{\langle n_{\alpha,\tau},m\rangle} z^m\,,
\nonumber
\end{align}
where $n_{\alpha,\tau}$ is the unique primitive element in $N$ such that $\fod \subset n_{\alpha,\tau}^\perp$ for all walls $\fod$ containing $\alpha(\tau)$, and 
$\langle n_{\alpha,\tau}, \alpha'(\tau) \rangle <0$. 
The path ordered automorphism 
$\mathfrak{p}_{\foD, \alpha}$ is the product of automorphisms
$\mathfrak{p}_{\foD, \alpha, \tau}$ for all $\tau \in [0,1]$ such that $\alpha(\tau)\in \Supp(\foD)$, and ordered following increasing values of $\tau$.
Finally, a scattering diagram is called \emph{consistent} 
if for any path $\alpha$ in $M_\RR \setminus \mathrm{Sing}(\foD)$ with $\alpha(0)=\alpha(1)$ 
the associated path ordered automorphism 
$\mathfrak{p}_{\foD, \alpha}$ is the identity.

Let $\mathbf{s}=(N, (e_i)_{i\in I}, \omega)$ be a symplectic seed and $\Sigma$ be a fan as in \S\ref{sec:cluster}. 
For every $i \in I$, we denote by $\Sigma_i$ the fan in $M_\RR /\RR v_i$ defined by 
\begin{equation} \label{eq_S_i}
\Sigma_i := \{ (\sigma + \RR v_i)/\RR v_i \,|\, \sigma \in \Sigma\,, v_i \in \sigma \} \,.
\end{equation}
It is the fan of the toric variety $D_i \simeq \overline{D}_i$.
For every codimension one cone $\underline{\rho} \in \Sigma_i$, there exists a unique codimension one cone $\rho \in \Sigma$ containing $v_i$ such that 
$\underline{\rho}=(\rho + \RR v_i)/\RR v_i$.
Finally, we denote by $\widetilde{\Sigma}_i^{[1]}$ the set of codimension one cones $\underline{\rho} \in \Sigma_i$
such that $\rho \subset e_i^{\perp}$.

\begin{definition}
The \emph{initial HDTV scattering diagram} $\foD_{\mathbf{s},\Sigma,\mathrm{in}}$
of a symplectic seed $\mathbf{s}$ and a fan $\Sigma$ as in \S \ref{sec:cluster} is the scattering diagrams in $M_\RR$ with set of walls
\begin{equation} \label{eq_D_in}
\foD_{\mathbf{s},\Sigma,\mathrm{in}} := \{(\rho, 1+t_i z^{v_i})\}_{i\in I, \underline{\rho} \in \widetilde{\Sigma}_i^{[1]}} \,.\end{equation}
\end{definition}

Note that all walls of $\foD_{\mathbf{s},\Sigma,\mathrm{in}}$ are incoming as in Definition \ref{def_incoming_wall} because $v_i \in \rho$ for all $\underline{\rho} \in \widetilde{\Sigma}_i^{[1]}$. 
By \cite[Thm 5.6]{HDTV}, there exists a consistent scattering diagram 
\begin{equation} \nonumber
\foD_{\mathbf{s},\Sigma}
\end{equation}
in $M_\RR$
containing the initial HDTV scattering diagram $\foD_{\mathbf{s},\Sigma,\mathrm{in}}$, and such that every wall $\fod \in \foD_{\mathbf{s},\Sigma} \setminus  \foD_{\mathbf{s},\Sigma,\mathrm{in}}$ is non-incoming. Moreover, such a scattering diagram is unique up to equivalence. We refer to $\foD_{\mathbf{s},\Sigma}$ as the \emph{HDTV scattering diagram} of a symplectic seed $\mathbf{s}$ and a fan $\Sigma$ as in \S \ref{sec:cluster}.

For every point $x\in M_\RR \setminus \Sing(\foD_{\mathbf{s},\Sigma})$, we denote
\begin{equation}
\label{Eq: fout}
f^{\mathrm{out}}_{\foD_{\mathbf{s},\Sigma},x} := \prod_{\substack{\fod \in \foD_{\mathbf{s},\Sigma}\setminus \foD_{\mathbf{s}, \Sigma, \mathrm{in}}\\ x \in \fod}} f_\fod \,,
\end{equation}
where the product is over all the walls of the HDTV scattering diagram containing $x$ and which are not initial walls, and 
\begin{equation}
\label{Eq: fin}
f^{\mathrm{in}}_{\foD_{\mathbf{s},\Sigma},x} := \prod_{\substack{\fod \in \foD_{\mathbf{s},\Sigma,\mathrm{in}}\\ x \in \fod}} f_\fod \,,
\end{equation}
where the product is over all initial walls containing $x$.

\subsection{Scattering calculation of punctured Gromov--Witten invariants} \label{sec: scattering_calculation}
In this section, we state and prove Theorem \ref{thm_hdtv}, a version of the main result of \cite{HDTV} relating HDTV scattering diagrams and punctured Gromov--Witten invariants.
Let $(X,D)$ be a $d$-dimensional log Calabi-Yau compactification of a cluster variety satisfying assumptions \eqref{eq_cluster_assumption}-\eqref{eq_H_int} of \S\ref{sec:cluster}.
Recall that $(X,D)$ is obtained as a blow-up of a toric variety $(X_\Sigma, D_\Sigma)$ and that we denote by $(B, \mathscr{P})$ be the tropicalization of $(X,D)$.

Each cone of the fan $\Sigma$ in $M_\RR$ corresponds to a stratum of $(X_\Sigma, D_\Sigma)$, whose strict transform in $(X,D)$
is a stratum
corresponding to an isomorphic cone of $\mathscr{P}$ in $B$. Gluing these identifications between cones of $\Sigma$ and cones of $\mathscr{P}$, we obtain as in 
\cite[\S 6]{HDTV} a piecewise-linear isomorphism
\begin{align} \label{eq_Upsilon}
\Upsilon:(M_{\RR},\Sigma)&\xlongrightarrow{\sim} (B,\mathscr{P}) \\
x &\longmapsto \Upsilon(x) \nonumber
\end{align}
between topological spaces endowed with a decomposition in integral cones.

We introduce notations that will be used to describe incoming walls of HDTV scattering diagrams in Theorem \ref{thm_hdtv}. 
For every $i \in I$, let $\mathcal{E}_i$ be the exceptional divisor in $X$ obtained as the preimage of $H_i$ by the blow-up morphism 
\begin{equation}\label{eq_bl}
\mathrm{Bl}: X  \longrightarrow X_\Sigma\,.\end{equation}
The projection $\mathcal{E}_i \rightarrow H_i$ is a $\PP^1$-fibration and we denote by $[E_i] \in N_1(X)$ the class of a $\PP^1$-fiber.
Every $\underline{\rho} \in \widetilde{\Sigma}_i^{[1]}$ as below \eqref{eq_S_i}
defines a 0-dimensional toric stratum $x_\rho$ of the toric variety $H_i$, and we denote by $E_\rho$ the $\PP^1$-fiber of $\mathcal{E}_i \rightarrow H_i$ above the point $x_\rho$.
The tropicalization of $E_\rho \simeq \PP^1$ endowed with the restriction of the log structure of $X$ is a wall type $\tau_{\rho}$
as in Definition \ref{def_wall_type}, where the graph $G$ consists of a unique vertex adjacent to a single leg, $Q_{\tau_\rho,\RR}^{\vee} = \underline{\rho}$, $u_{\tau_{\rho}} = v_i$ and \[\bigcup_{s \in Q_{\tau_\rho,\RR}^{\vee}} h_s(L_{\mathrm{out}})=\Upsilon(\rho)\,.\]

We now introduce notations that will be used to describe non-incoming walls of HDTV scattering diagrams in Theorem \ref{thm_hdtv}.
Fix $\mathbf{A}=(a_i)_{i\in I}\in \NN^I$
and a point $x \in \Supp(\foD_{\mathbf{s},
\Sigma}) \setminus \Sing(\foD_{\mathbf{s},\Sigma})$.
Then, there exists a unique wall $\fod_x$ of $\foD_{\mathbf{s},
\Sigma}$ containing $x$, and we denote by $W_x$ the unique hyperplane in $M_\RR$ containing $\fod_x$. 
We define
\[
m_{\mathbf{A}}:= -\sum_{i\in I}a_i v_i \in M\,,
\]
\[
\mathcal{A}_x := \{ \mathbf{A}=(a_i)_{i\in I}\in \NN^I \,|\, m_{\mathbf{A}} \in W_x \} \subset \NN^I\,,
\]
and we denote by $\mathcal{T}_{\mathbf{A}}^x$ the set of 
wall types $\tau$ as in Definition \ref{def_wall_type}
such that $x \in h(\tau_{\mathrm{out}})$ and $\Upsilon_{*} m_{\mathbf{A}} = u_{\tau}$.

Let $\sigma_x$ be the smallest cone of $\Sigma$ containing $x$. As $x \notin \Sing(\foD_{\mathbf{s},
\Sigma})$, it follows that $\sigma_x$ is of codimension one or zero in $M_\RR$ and the linear span of $\sigma_x$ contains the hyperplane $W_x$.
As the fan $\Sigma$ is smooth, it follows that $\sigma_x$ is a regular cone, and so the 
primitive generators $m_1^{\sigma_x}, \dots, m_{\ell}^{\sigma_x}$
of the rays of $\sigma_x$ form a basis of 
the $\Z$-linear span of $\sigma_x$.
Therefore, for every $\mathbf{A} \in \mathcal{A}_x$, there exists unique 
integers $b_1, \dots, b_\ell \in \Z$ such that 
\[ m_{\mathbf{A}}=\sum_{j=1}^{\ell} b_j m_j^{\sigma_x}\,.\]
For every $1 \leq j \leq \ell$, let $D_j^{\sigma_x}$ be the toric divisor of 
$X_\Sigma$ corresponding to the ray
$\RR_{\geq 0} m_j^{\sigma_x}$ of $\Sigma$.
By standard toric geometry, there  exists a unique curve class $\bar{\beta}_{\mathbf{A}}^x \in N_1(X_\Sigma)$
such that, for every toric divisor $D'$ of $X_\Sigma$, we have 
\[ \bar{\beta}_{\mathbf{A}}^x \cdot D' =\sum_{i \in I_{D'}} a_i 
+ \sum_{j \in J_{D'}} b_j\,, \]
where
\[ I_{D'} = \{ i \in I \,|\, D_i=D'\}\,\,\,\, \text{and}\,\,\,\,\,J_{D'}=\{ 1\leq j \leq \ell\,|\, D_j^\sigma=D'\}\,.\]
Finally, define the curve class
\begin{equation} \label{eq_curve_class}
\beta_{\mathbf{A}}^x := \mathrm{Bl}^{*} \bar{\beta}_{\mathbf{A}}^x 
- \sum_{i\in I} a_i [E_i] \in N_1(X)\,.\end{equation}

\begin{theorem} \label{thm_hdtv}
Let $\mathbf{s}$ be a symplectic seed, $\Sigma$ a fan, and $(X,D)$ a log Calabi-Yau compactification of the corresponding cluster variety satisfying assumptions \eqref{eq_cluster_assumption} and \eqref{eq_H_int}.
Then, for every point $x \in \Supp(\foD_{\mathbf{s},\Sigma}) \setminus \Sing(\foD_{\mathbf{s},\Sigma})$, the functions $f_{\foD_{\mathbf{s},\Sigma},x}^{\mathrm{in}}$ and $f_{\foD_{\mathbf{s},\Sigma}, x}^{\mathrm{out}}$ defined as in \eqref{Eq: fin} and \eqref{Eq: fout} respectively are given as follows:
\begin{equation}\nonumber
f_{\foD_{\mathbf{s},\Sigma},x}^{\mathrm{in}}
=
\exp \left(
\sum_{\substack{i\in I, \underline{\rho} \in \widetilde{\Sigma}_i^{[1]}\\ x\in \rho}} \sum_{\ell \geq 1} \ell k_{\tau_\rho} N_{\tau_{\rho}, \ell}^{(X,D)} z^{\ell v_i} t_i^{\ell}
\right)
\end{equation}
and
\begin{equation} \nonumber
f_{\foD_{\mathbf{s},\Sigma}, x}^{\mathrm{out}}
= \exp \left(  
\sum_{\mathbf{A}=(a_i)_i \in \mathcal{A}_x}
\sum_{\tau \in \mathcal{T}_{\mathbf{A}}^x} 
k_\tau
N_{\tau,\beta_{\mathbf{A}}^x}^{(X,D)}  z^{\sum_{i\in I} a_i v_i} \prod_{i \in I} t_i^{a_i}
\right) \,.\end{equation}
where $ N_{\tau_{\rho}, \ell[E_i]}^{(X,D)}$ and $N_{\tau,\beta_{\mathbf{A}}^x}^{(X,D)} $ are punctured log Gromov--Witten invariants of $(X,D)$ defined as in \eqref{eq_gw} and $k_{\tau_\rho}, k_\tau$ are coefficients defined as in \eqref{eq_coeff}.
\end{theorem}

\begin{proof}
We explain how Theorem \ref{thm_hdtv} follows from \cite[Thm $6.1$]{HDTV}. By \cite[Def 3.10]{gross2021canonical}, for any log Calabi-Yau pair $(X,D)$, the punctured Gromov--Witten invariants $N_{\tau,\beta}^{(X,D)}$ define a \emph{canonical scattering diagram} $\foD_{(X,D)}$ in the tropicalization $(B,\mathscr{P})$ of $(X,D)$. We refer to \cite[\S 3.2]{gross2021canonical} for the general notion of scattering diagram (called wall structure in \cite{gross2021canonical}) in $(B,\mathscr{P})$.
For log Calabi-Yau pairs $(X,D)$ obtained as blow-up of a toric variety $X_\Sigma$ with fan $\Sigma$ in $M_\RR$ along disjoint smooth hypersurfaces $H=\cup_{i \in I} H_i$ of its toric boundary, an initial scattering diagram $\foD_{(X_\Sigma, H), \mathrm{in}}$
in $M_\RR$ is defined explicitly in \cite[Eq.\,(5.6)]{HDTV}.
A scattering diagram $\foD_{(X_\Sigma,H)}$ in $M_\RR$ is then defined using \cite[Thm 5.6]{HDTV} as the unique scattering diagram in $M_\RR$ containing $\foD_{(X_\Sigma, H), \mathrm{in}}$ and such that the complement $\foD_{(X_\Sigma, H)} \setminus \foD_{(X_\Sigma, H), \mathrm{in}}$ consists only of non-incoming walls. Then, the statement of \cite[Thm $6.1$]{HDTV} is the equivalence of the canonical scattering diagram $\foD_{(X,D)}$ with a scattering diagram $\Upsilon(\foD_{(X_\Sigma,H)})$ constructed from $\Upsilon(\foD_{(X_\Sigma,H)})$ in an explicit way described in \cite[\S 6]{HDTV} using the map $\Upsilon$ as in \eqref{eq_Upsilon}.

By \cite[Lemma 4.3]{ABclustermirror}, when $(X,D)$ is a log Calabi--Yau compactifications of a cluster variety, that is when the hypersurfaces $H_i$ are of the form given in 
\eqref{eq_H_i}, the initial scattering diagram $\foD_{(X_\Sigma, H), \mathrm{in}}$ in \cite[Eq.\,(5.6)]{HDTV} coincides with the initial scattering diagram $\foD_{\mathbf{s},\Sigma,\mathrm{in}}$
in \eqref{eq_D_in}. Thus, the scattering diagrams $\foD_{(X_\Sigma, H)}$ and $\foD_{\mathbf{s},\Sigma}$ are equivalent in this case by the uniqueness result in \cite[Thm 5.6]{HDTV}. Hence, it follows from \cite[Thm 6.1]{HDTV} that the canonical scattering diagram $\foD_{(X,D)}$ is equivalent to $\Upsilon(\foD_{\mathbf{s},\Sigma})$. 
Theorem \ref{thm_hdtv} then follows from the definition of the canonical scattering diagram $\foD_{(X,D)}$ in terms of punctured Gromov--Witten invariants $N_{\tau,\beta}^{(X,D)}$ (see \cite[Def 3.10]{gross2021canonical} or \cite[Def 2.33]{HDTV}) and from the explicit description of the map $\foD_{\mathbf{s},\Sigma} \mapsto\Upsilon(\foD_{\mathbf{s},\Sigma})$ given in \cite[\S 6]{HDTV}.
\end{proof}

\section{DT/punctured GW correspondence}

\subsection{Quiver and cluster compatibility} 
\label{sec_compatibility}

\begin{definition} \label{def_compatibility}
Let $Q$ be a quiver and let $\mathbf{s}=(N,(e_i)_{i\in I}, \omega)$ be a symplectic seed with index set $I=\{ i\in Q_0\,|\, \iota_{s_i} \omega_Q \neq 0\}$ as in \S \ref{sec:cluster}. 
A \emph{compatibility data} between $Q$ and $\mathbf{s}$ is a linear map
\begin{equation} \label{eq_psi}
\psi \colon N_Q \longrightarrow N
\end{equation}
with finite cokernel, that is, with $\psi \otimes \Q$ is surjective, 
such that $\psi(s_i)=e_i$ for all $i \in I$, and $\psi^{*}\omega=\omega_Q$, that is, $\omega(e_i,e_j)=\omega_Q(s_i,s_j)$ for all $i,j \in I$. 
\end{definition}

\begin{remark}
By Definition \ref{def_seed} of a symplectic seed, $\omega \otimes \Q$ is non-degenerate, and so the rank of $\omega$ is equal to the rank of $N$. As $\psi \otimes \Q$ is surjective and $\psi^{\star} \omega =\omega_Q$, it follows that the rank of $N$ is equal to the rank of $\omega_Q$. In particular, the dimension of the cluster variety defined by the symplectic seed $\mathbf{s}$ is equal to the rank of the skew-symmetrized Euler form $\omega_Q$.
\end{remark}

In what follows we provide examples of compatibility data.

\begin{example}
    Let $Q$ be a quiver. 
    Let $\mathbf{s}=(N,(e_i)_{i\in I}, \omega)$
    be the symplectic seed defined by $N=N_Q/
    \ker \omega_Q$, $e_i=\pi(s_i)$ for all $i \in I$, where $\psi: N_Q \rightarrow N$ is the quotient map, and $\omega$ the non-degenerate skew-symmetric form on $N$ induced by $\omega_Q$. Then, $\psi$ is a compatibility data between $Q$ and $\mathbf{s}$. 
\end{example}

\subsection{Comparison of scattering diagrams}

Let $Q$ be a quiver, $\mathbf{s}=(N,(e_i)_{i\in I},\omega)$ a symplectic seed, and $\psi: N_Q \rightarrow N$ a compatibility data between $\mathbf{s}$ and $Q$ as in \eqref{eq_psi}. We denote by
\begin{equation}\nonumber
\psi^{\vee}: M_\RR \longrightarrow M_{Q,\RR}
\end{equation}
the dual map of $\psi$.

Let $\foD_Q$ be a consistent $Q$-scattering diagram in $M_{Q,\RR}$ without central walls, as in Definitions \ref{def_dt_scattering}-\ref{def_central}. 
We define a scattering diagram $(\psi^\vee)^\star \foD_Q = \{(\psi^\vee)^\star (\fod, f_\fod)\}$ in $M_\RR$ in the sense of Definition \ref{def_scattering} as follows:
for every wall $(\fod, f_\fod)$ of $\foD_Q$, we define 
$(\psi^\vee)^\star (\fod, f_\fod):= ((\psi^\vee)^{-1}(\fod), (\psi^\vee)^\star f_\fod) $.
Here, 
$(\psi^\vee)^{-1}(\fod) \subset M_\RR$ is the preimage by $\psi^\vee$ of $\fod \subset M_{Q,\RR}$.
Note that as $(\fod, f_\fod)$ is not central, the primitive direction $\gamma_\fod \in N^+_Q$ of $\fod$ is not in $\ker \omega_Q$, and so, as $\psi^{\star} \omega =\omega_Q$ by Definition \ref{def_compatibility}, we have $\psi(\gamma_\fod) \notin \ker \omega$. In particular, $\psi(\gamma_\fod) \neq 0$, and so, as $(\psi^\vee)^{-1}(\fod)$ is contained in $\psi(\gamma_\fod)^\perp$, one deduces that $(\psi^\vee)^{-1}(\fod)$ is of codimension one in $M_\RR$ if non empty.

Now to define $(\psi^\vee)^\star f_\fod$, first observe that one can write:
\begin{equation} \label{eq_f}
f_\fod = 1+\sum_{k \geq 1} c_k z^{k \gamma_\fod}\end{equation}
with $c_k \in \Q$. Decompose $\gamma_\fod \in N_Q^+$ in the basis $(s_i)_{i \in Q_0}$:
\[ \gamma_{\fod}=\sum_{i \in Q_0} \gamma_{\fod,i} s_i\]
with $\gamma_{\fod,i} \in \Z_{\geq 0}$.
Note that as $\foD_Q$ is without central walls, it follows from Remark \ref{rem_ker} that if $c_k \neq 0$, then $\gamma_{\fod,i} =0$ if $i \notin I$, and so $\gamma_{\fod}=\sum_{i \in I} \gamma_{\fod,i} s_i$, and $\psi(\gamma_{\fod})=\sum_{i\in I} \gamma_{\fod,i} e_i$.
Finally, we set 
\begin{equation}\label{eq_psi_f}
     (\psi^\vee)^\star f_\fod := 1+\sum_{k \geq 1} 
c_k  z^{k \iota_{\psi(\gamma_\fod)} \omega}
\left( \prod_{i \in I} t_i^{k \gamma_{\fod,i}} \right) \,.
\end{equation}
Note that the finiteness condition in Definition \ref{def_scattering} for $(\psi^\vee)^\star \foD_Q$ follows from the finiteness condition in Definition \ref{def_dt_scattering} for $\foD_Q$. Hence, $(\psi^\vee)^\star \foD_Q$ is indeed a scattering diagram.

\begin{lemma} \label{lem_psi_consistency}
Let $\foD_Q$ be a consistent $Q$-scattering diagram in $M_\RR$ without central walls. Then 
$(\psi^\vee)^\star \foD_Q$ is also consistent.
\end{lemma}

\begin{proof}
Let $\alpha: [0,1] \rightarrow M_\RR \setminus \Sing ((\psi^\vee)^\star \foD_Q)$ be a continuous path intersecting $\Supp((\psi^\vee)^\star \foD_Q)$ transversally. Then $\psi^\vee \circ \alpha$ is a path in $M_{Q,\RR} \setminus \Sing (\foD_Q)$  intersecting $\Supp(\foD_Q)$ transversally.
In order to prove Lemma \ref{lem_psi_consistency}, it is enough to show that if the automorphism $\mathfrak{p}_{\foD_Q, \psi^\vee \circ \alpha}$ of $\QQ[\![N_Q^+]\!]$ defined in \eqref{eq_autom_Q} is the identity, then the automorphism 
$\mathfrak{p}_{(\psi^\vee)^\star\foD_Q, \alpha}$
of $\Q[M][\![(t_i)_{i\in I}]\!]$
defined in \eqref{eq_autom} is also the identity.
It is sufficient to prove this result for every finite truncation with respect to powers of the maximal ideals $\mathfrak{m}$ of $\QQ[\![N_Q^+]\!]$ and $\Q[\![(t_i)_{i\in I}]\!]$. In other words, for every $k \geq 1$, we consider the induced automorphisms $\mathfrak{p}^{\leq k}_{\foD_Q, \psi^\vee \circ \alpha}$ and $\mathfrak{p}^{\leq k}_{(\psi^\vee)^\star\foD_Q, \alpha}$
acting on $\QQ[N_Q^+]/\mathfrak{m}^k$ and $\Q[M][(t_i)_{i\in I}]/\mathfrak{m}^k$ respectively.

Denote by $M'$ the image in $M$ of the map
\begin{align} N_Q^+ &\longrightarrow M \\
\gamma &\longmapsto \iota_{\psi(\gamma)} \omega \,. \nonumber
\end{align}
As $N_Q^+$ is a cone of maximal rank in $N_Q$, $\psi \otimes \QQ$ is surjective by Definition \ref{def_compatibility} and $\omega \otimes \QQ$ is non-degenerate by Definition \ref{def_seed}, it follows that $M'$ is also a cone of maximal rank in $M$. As $\mathfrak{p}^{\leq k}_{(\psi^\vee)^\star\foD_Q, \alpha,x}$ is an algebra endomorphism and $\mathfrak{p}^{\leq k}_{(\psi^\vee)^\star\foD_Q, \alpha,x}= \mathrm{Id} \mod \mathfrak{m}$, it is enough to prove that $\mathfrak{p}^{\leq k}_{(\psi^\vee)^\star\foD_Q, \alpha,x}$ is the identity on $\Q[M'][(t_i)_{i\in I}]$ in order to conclude that it is the identity on $\Q[M][(t_i)_{i\in I}]$.

It follows from the definition of $M'$ that the additive map 
\begin{align}
 N_Q^{+} &\longrightarrow M' \oplus \Z^I \\
\gamma=\sum_{i\in Q_0} \gamma_i s_i 
&\longmapsto (\iota_{\psi(\gamma)}\omega, (\gamma_i)_{i\in I})  \nonumber
\end{align}
is surjective, and so induces a surjective algebra morphism 
\[ \varphi : \Q[N_Q^+]/\mathfrak{m}^k \longrightarrow \Q[M'][(t_i)_{i\in I}]/\mathfrak{m}^k \,.\]
Comparing the explicit formulas \eqref{eq_f}-\eqref{eq_autom_Q} for $\foD_Q$ with \eqref{eq_psi_f}-\eqref{eq_autom} for $(\psi^\vee)^\star \foD_Q$, one checks that 
\[ \mathfrak{p}^{\leq k}_{(\psi^\vee)^\star\foD_Q, \alpha} \circ \varphi = \varphi \circ \mathfrak{p}^{\leq k}_{\foD_Q, \psi^\vee \circ \alpha}\,.\]
In particular, as $\varphi$ is surjective, if we have
$\mathfrak{p}^{\leq k}_{\foD_Q, \psi^\vee \circ \alpha}=\mathrm{Id}$, then we also have $\mathfrak{p}^{\leq k}_{(\psi^\vee)^\star\foD_Q, \alpha}=\mathrm{Id}$, and this concludes the proof.
\end{proof}

Using the above construction $\foD_Q \mapsto (\psi^\vee)^\star \foD_Q$, we compare in the following result the stability scattering diagram $\overline{\foD}_{(Q,W)}^{\mathrm{st}}$ defined in 
\ref{def_dt_scattering}
and the HDTV scattering diagram $\foD_{\mathbf{s,\Sigma}}$ defined in \S \ref{sec: hdtv}.

\begin{theorem} \label{thm_comparison_scattering}
Let $(Q,W)$ be a quiver with potential having trivial attractor DT invariants, $\mathbf{s}=(N,(e_i)_{i\in I},\omega)$ a symplectic seed,  and $\Sigma$ a fan as in \S\ref{sec:cluster}. Fix a compatibility data $\psi: N_Q \rightarrow N$  between $\mathbf{s}$ and $Q$ as in \eqref{eq_psi}.
Then, the 
$Q$-scattering without central walls $\overline{\foD}_{(Q,W)}^{\mathrm{st}}$, obtained from the stability scattering diagram $\overline{\foD}_{(Q,W)}^{\mathrm{st}}$, and the HDTV scattering diagram $\foD_{\mathbf{s,\Sigma}}$
are related as follows:
\[ (\psi^\vee)^\star \overline{\foD}_{(Q,W)}^{\mathrm{st}}= \foD_{\mathbf{s},\Sigma} \,.\]
\end{theorem}

\begin{proof}
By Theorem \ref{thm_dt_cluster}, the $Q$-scattering diagram 
$\overline{\foD}_{(Q,W)}^{\mathrm{st}}$ and the cluster scattering diagram $\foD_Q^{\mathrm{cl}}$ are equivalent.
Hence, it is enough to show that $(\psi^\vee)^\star  \foD_Q^{\mathrm{cl}}$ and $\foD_{\mathbf{s},\Sigma}$ are equivalent.
Comparing the explicit descriptions \eqref{eq_initial_cluster} and \eqref{eq_D_in} of the initial scattering diagrams 
$\foD_{Q, \mathrm{in}}^{\mathrm{cl}}$
and $\foD_{s,\Sigma,\mathrm{in}}$ using that $v_i = \iota_{e_i} \omega = \iota_{\psi(e_i)} \omega$ for all $i \in I$, one obtains that
$(\psi^\vee)^\star  \foD_Q^{\mathrm{cl}, \mathrm{in}}$ is equivalent to a scattering diagram containing 
$\foD_{s,\Sigma,\mathrm{in}}$, and whose walls not in $\foD_{s,\Sigma,\mathrm{in}}$ are non-incoming. On the other hand, $\foD_Q^{\mathrm{cl}}$ is consistent and so 
$(\psi^\vee)^\star  \foD_Q^{\mathrm{cl}, \mathrm{in}}$ is also consistent by Lemma \ref{lem_psi_consistency}. It follows that $(\psi^\vee)^\star  \foD_Q^{\mathrm{cl}, \mathrm{in}}$ is equivalent to a consistent scattering diagram containing $\foD_{s,\Sigma,\mathrm{in}}$ and whose walls not in $\foD_{s,\Sigma,\mathrm{in}}$ are non-incoming. Therefore,$(\psi^\vee)^\star  \foD_Q^{\mathrm{cl}, \mathrm{in}}$ is equivalent to $\foD_{s,\Sigma}$ by the uniqueness result in \cite[Thm 5.6]{HDTV}.
\end{proof}

\subsection{DT/punctured GW correspondence}
\label{sec_dt_gw}
In this section, we prove our main result, Theorem \ref{thm_main} below, relating quiver DT invariants and punctured Gromov--Witten invariants.

Let $(Q,W)$ be a quiver with potential having trivial attractor DT invariants. 
Let $\mathbf{s}=(N,(e_i)_{i\in I},\omega)$ a symplectic seed, $\Sigma$ a fan as in \S\ref{sec:cluster}, and $(X,D)$ a log Calabi-Yau compactification of the corresponding cluster variety. 
Let \[\psi: N_Q \longrightarrow N\] be a compatibility data between $\mathbf{s}$ and $Q$ as in \eqref{eq_psi}.
Let $\gamma=(\gamma_i)_{i\in I} \in
N_Q\setminus \ker \omega_Q$ be a dimension vector 
and $\theta \in \psi^\vee (M_\RR) \cap \gamma^{\perp}\subset M_{Q,\RR}$ be a general stability parameter contained in $\psi^\vee (M_\RR)$. Let $x \in M_\RR$ be a general point such that $\psi^\vee(x)=\theta$. Then, viewing $\gamma=(\gamma_i)_{i \in I}$ as an element of $\NN^I$, one defines a curve class $\beta_\gamma^x \in N_1(X)$ and a set of wall types $\mathcal{T}_\gamma^x$ as in \S\ref{sec: scattering_calculation}
by taking $\mathbf{A}=\gamma$.

\begin{theorem} \label{thm_main}
Let $(Q,W)$ be a quiver with potential having trivial attractor DT invariants, $\mathbf{s}=(N,(e_i)_{i\in I}, \omega)$ a  symplectic seed with index set \[I=\{ i\in Q_0\,|\, \iota_{s_i} \omega_Q \neq 0\}\,,\] and $(X,D)$ a log Calabi-Yau compactification of the corresponding cluster variety satisfying assumptions \eqref{eq_cluster_assumption} and \eqref{eq_H_int}. 
Fix a compatibility data $\psi \colon N_Q \rightarrow N$  between $\mathbf{s}$ and $Q$ as in \eqref{eq_psi},
and let $\gamma \in N_Q \setminus \ker \omega_Q$ be a dimension vector such that $\gamma \notin \Z_{\geq 1}s_i$ for all $i \in I$. Then, for every general stability parameter $\theta \in \psi^\vee (M_\RR) \cap \gamma^{\perp}\subset M_{Q,\RR}$ and point $x \in M_\RR$ such that $\psi^\vee(x)=\theta$, we have 
\begin{equation} \label{eq_main} 
\overline{\Omega}_\gamma^{+,\theta} 
= \frac{1}{|\gamma|}\sum_{\tau \in \mathcal{T}_\gamma^x} k_\tau N_{\tau,\beta_\gamma^x}^{(X,D)} \,,\end{equation}
    where $\overline{\Omega}_\gamma^{+,\theta} $ is the rational DT invariant of $(Q,W)$ as in \eqref{eq_dt_rational}, and   $N_{\tau,\beta_\gamma^x}^{(X,D)}$ are the punctured Gromov--Witten invariants of $(X,D)$ as in \eqref{eq_gw}.
\end{theorem}

\begin{proof}
Denote $\gamma_0 = \frac{\gamma}{|\gamma|}$.
By \eqref{eq_stab_scattering}, we have 
\[ f_{\overline{\foD}^{\mathrm{st}}_{(Q,W)}, \theta} = \exp \left( \sum_{\gamma' \in \Z_{\geq 0} \gamma_0} |\gamma'| \,\, \overline{\Omega}_{\gamma'}^{+,\theta} z^{\gamma'} \right) \,.\]
Applying the definition \eqref{eq_psi_f} of $(\psi^\vee)^\star$, we obtain
\[ (\psi^\vee)^\star f_{\overline{\foD}^{\mathrm{st}}_{(Q,W)}, \theta}
=\exp \left( \sum_{\gamma'=(\gamma'_i)_{i\in I} \in \Z_{\geq 1} \gamma_0} |\gamma'|\,\, \overline{\Omega}_{\gamma'}^{+,\theta} z^{\iota_{\gamma'} \omega} \prod_{i \in I} t_i^{\gamma'_i} \right)\,.
\]
Using that $\gamma'=(\gamma'_i)_{i\in I}=\sum_{i\in I}\gamma'_i s_i$, we have $\iota_{\gamma'} \omega=\sum_{i\in I} \gamma'_i v_i$ and so 
\begin{equation} \label{eq_proof_1}
(\psi^\vee)^\star f_{\overline{\foD}^{\mathrm{st}}_{(Q,W)}, \theta}
=\exp \left( \sum_{\gamma'=(\gamma'_i)_{i\in I} \in \Z_{\geq 1} \gamma'_0} |\gamma'|\,\, \overline{\Omega}_{\gamma'}^{+,\theta} z^{\sum_{i\in I} \gamma'_i v_i} \prod_{i \in I} t_i^{\gamma'_i} \right)\,.
\end{equation}
By Theorem \ref{thm_comparison_scattering}, we have  $(\psi^\vee)^\star \overline{\foD}_{(Q,W)}^{\mathrm{st}}= \foD_{\mathbf{s}}$ and so 
\[(\psi^\vee)^\star f_{\overline{\foD}^{\mathrm{st}}_{(Q,W)}, \theta} =f_{\foD_{\mathbf{s}}, x}= f_{\foD_{\mathbf{s}}, x}^{\mathrm{in}} f_{\foD_{\mathbf{s}}, x}^{\mathrm{out}}\,. \]
Combining \eqref{eq_proof_1} with Theorem \ref{thm_hdtv}, we obtain
\begin{align} &\sum_{\gamma'=(\gamma'_i)_{i\in I} \in \Z_{\geq 1} \gamma_0} |\gamma'|\,\, \overline{\Omega}_{\gamma'}^{+,\theta} z^{\sum_{i\in I} \gamma_i' v_i} \prod_{i \in I} t_i^{\gamma'_i} \\&=\sum_{\substack{i\in I, \rho \in \Sigma_i\\ x\in \rho}} \sum_{\ell \geq 1} k_\rho N_{\rho, \ell[E_i]}^{(X,D)} z^{\ell v_i} t_i^{\ell [E_i]}
+ 
\sum_{\mathbf{A}=(a_i)_i \in \mathcal{A}_x}
\sum_{\tau \in \mathcal{T}_{\mathbf{A}}^x} 
k_\tau
N_{\tau,\beta_{\mathbf{A}}^x}^{(X,D)}  z^{\sum_{i\in I} a_i v_i} \prod_{i \in I} t_i^{a_i}    \nonumber
\end{align}
It remains to isolate on both sides the terms proportional $z^{\sum_{i\in I} \gamma_i v_i} \prod_{i \in I} t_i^{\gamma_i}$. As we are assuming that $\gamma \notin \Z_{\geq 1}s_i$ for all $i\in I$, it follows that $z^{\sum_{i\in I} \gamma_i v_i} \prod_{i \in I} t_i^{\gamma_i}$ is not equal to $z^{\ell v_i} t_i^{\ell [E_i]}$ for any $i\in I$ and $\ell \geq 1$. Hence, we conclude that 
\[ |\gamma|\,\, \overline{\Omega}_\gamma^{+,\theta} 
= \sum_{\tau \in \mathcal{T}_{\gamma}^x} 
k_\tau
N_{\tau,\beta_{\gamma}^x}^{(X,D)}\,,\]
and so \eqref{eq_main} follows.
\end{proof}

\section{Examples: local $\PP^2$ and cubic surfaces}

In this section, we give two examples to illustrate Theorem \ref{thm_main}.

\subsection{The local projective plane}
\label{sec_ex_local}

Let $Q$ be the $3$-node quiver in Figure \ref{Fig:3quiver}. The skew-symmetric form $\omega_Q$ on $N_Q=\Z s_1 \oplus \Z s_2 \oplus \Z s_3$ satisfies \[\omega_Q(s_1,s_2)=\omega_Q(s_2,s_3)=\omega_Q(s_3, s_1)=3\,.\] 
Moreover, $\omega_Q$ is of rank two and its kernel is given by $\ker \omega_Q =\Z (s_1+s_2+s_3)$. 

\begin{figure}[h]
\center{\includegraphics{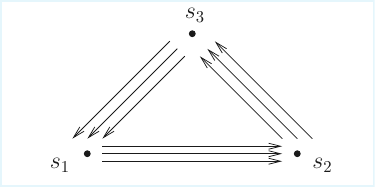}}
\caption{The quiver attached to the local projective plane.}
\label{Fig:3quiver}
\end{figure}

Let $\mathbf{s}=(N, (e_i)_{1\leq i\leq 3}, \omega)$ be the symplectic seed defined by $N=\Z^2$, 
\[ e_1=(1,1)\,, e_2=(-2,1)\,, e_3=(1,-2)\,,\] and $\omega(-,-)=\det(-,-)$.
Then, the map $\psi: N_Q \rightarrow N$ defined by $\psi(s_i)=e_i$ for all $1\leq i \leq 3$ is a compatibility data in the sense of Definition \ref{def_compatibility}. 
To construct the corresponding cluster variety, note that $\omega$ is non-degenerate, and so we can use it to identify $M$ with $N$ and $v_i=\iota_{e_i} \omega$ with $e_i$. Therefore, we consider the fan $\Sigma$ in $\RR^2$ consisting of the three rays $\RR_{\geq 0}e_1$, 
$\RR_{\geq 0}e_2$, $\RR_{\geq 0}e_3$. The corresponding toric surface $X_\Sigma$ has three toric $A_2$ cyclic quotient singularities, which can be resolved torically if one insists on having $X_\Sigma$ smooth and projective as in \S \ref{sec:cluster}. Let $X$ be the surface obtained by blowing-up a point on each of the toric divisors of $X_\Sigma$ corresponding to the rays $\RR_{\geq 0}e_1$, $\RR_{\geq 0}e_2$, $\RR_{\geq 0}e_3$, and $D$ be the strict transform of the toric boundary divisor of $X_\Sigma$. Then, $U=X \setminus D$ is the cluster variety defined by the symplectic seed $\mathbf{s}$, and $(X,D)$ is a log Calabi--Yau compactification of $U$. As $\coker \psi \simeq \Z/3\Z$, the holomorphic symplectic cluster variety $U$ is a $\Z/3\Z$-quotient of a symplectic fiber of the Poisson $\mathcal{X}$ cluster variety defined by $Q$.

Then, as reviewed for example in 
\cite{bousseau2022bps}, there exists a potential $W$ on $Q$ and an equivalence of triangulated categories 
\begin{equation} \label{eq_equiv}
\Phi: D^b \mathrm{Rep}(Q,W) \xlongrightarrow{\sim} D^b \mathrm{Coh}(K_{\PP^2})\end{equation}
between the derived category of representations of $(Q,W)$ and the derived category of coherent sheaves on the non-compact toric Calabi--Yau 3-fold given by the local projective plane $K_{\PP^2}=\mathcal{O}_{\PP^2}(-3)$, sending the three simple representations of $Q$ to the spherical objects $E_1 = \iota_\star \mathcal{O}[-1]$, $E_2 = \iota_\star(\Omega_{\PP^2}(1))$, $E_3 = \iota_\star(\mathcal{O}(-1))[1]$ in $D^b \mathrm{Coh}(K_{\PP^2})$, where $\iota: \PP^2 \hookrightarrow K_{\PP^2}$ is the inclusion of the zero section.

By \cite[Thm 1]{bousseau2022bps}, $(Q,W)$ has trivial attractor DT invariants, and so Theorem \ref{thm_main} applies. 
Theorem \ref{thm_main} computes $\overline{\Omega}_\gamma^{+,\theta}$ for $\theta \in M_\RR \subset M_{Q,\RR}$, that is, for $\theta=(\theta_1,\theta_2,\theta_3)$ such that $\theta_1+\theta_2+\theta_3=0$. For every $\gamma \in N_Q \setminus \ker \omega_Q$ such that $\gamma \notin \Z_{\geq 1}s_i$ for all $1\leq i \leq 3$, the intersection $\gamma^\perp \cap M_{\RR}$ is a line separated into two chambers: the attractor chamber $\RR_{\geq 0} \iota_\gamma \omega_Q$, and the anti-attractor chamber $-\RR_{\geq 0} \iota_\gamma \omega_Q$. If $\theta$ is in the attractor chamber, then the set of wall types $\mathcal{T}_\gamma^\theta$ is empty, and so, by Theorem \ref{thm_main}, we have
$\overline{\Omega}_\gamma^\theta=0$. If $\theta$ is in the anti-attractor chamber, then the set $\mathcal{T}_\gamma^\theta$ contains a single wall type $\tau_\gamma^\theta$, given by the half-line coming out of the origin and passing by $\Upsilon(\theta)$ in the tropicalization $B$ of $(X,D)$ -- see \cite[Ex 3.14]{gross2021canonical}. Moreover, we have $k_\tau=|\iota_{\psi(\gamma)} \omega|$ -- 
Therefore, by Theorem \ref{thm_main}, we have 
\begin{equation}\label{eq_01}
\overline{\Omega}_\gamma^{+,\theta} =\frac{|\iota_{\psi(\gamma)} \omega|}{|\gamma|} 
N_{\tau_\gamma^\theta, \beta_\gamma^\theta}^{(X,D)} \,.
\end{equation}

Using the equivalence \eqref{eq_equiv}, one can deduce a result comparing geometric DT invariants of $K_{\PP^2}$
and punctured Gromov--Witten invariants of $(X,D)$. For every $v=(r,d,\chi)\in \Z^3$, one can define a geometric DT invariant $\overline{\Omega}_v^{+}$ using the moduli space of Gieseker semistable coherent sheaves on $\PP^2$ of rank $r$, degree $d$, and Euler characteristic $\gamma$ -- see for example \cite{bousseau2022scattering}. 
By \cite[(5.4)]{bousseau2022bps}, under the equivalence 
\eqref{eq_equiv}, a coherent sheaf of class $v$ is mapped to an object of $D^b \mathrm{Rep}(Q,W)$ of dimension vector 
\[ \gamma(v):=(-\chi, r+d-\chi, r+2d-\chi)\,.\]
While there are in general walls in the space of Bridgeland stability conditions on $D^b \mathrm{Coh}(K_{\PP^2})$ separating the quiver DT invariants $\overline{\Omega}_{\gamma(v)}^{+,\theta}$ and the geometric DT $\overline{\Omega}_v^+$, these walls are absent for \emph{normalized} coherent sheaves, that is, with slope $\mu=\frac{d}{r}$ satisfying $-1<\mu\leq 0$:
we have 
\begin{equation} \label{eq_02}
\overline{\Omega}_{v}^{+} = \overline{\Omega}_{\gamma(v)}^{+,\theta} 
\end{equation}
for $\theta$ in the anti-attractor chamber $-\RR_{\geq 0} \iota_{\gamma(v)} \omega_Q$.
Indeed, in this case, moduli spaces of Gieseker semistable sheaves coincide with moduli spaces of $\theta$-semistable quiver representations by  
\cite[Prop 2.3]{DrP} and using \cite[\S 5.2, proof of Thm 1]{bousseau2022bps} to compare representations of $Q$ with representations of the Beilinson quiver obtained by removing from $Q$ all the three arrows between a given pair of vertices. A pictorial proof can also be obtained by looking in \cite[Figure 10]{bousseau2022bps} at the scattering diagram $\mathcal{D}_\psi$ defined in \cite{bousseau2022bps} for the phase $\psi=\frac{\pi}{2}$: the rays of $\mathcal{D}_0$ going out from the orbifold point, where the quiver description is valid, go directly to the large volume region, where the geometric description is valid, without any further scattering.
Therefore, we obtain the following result.

\begin{theorem} \label{thm_localP2}
For every $v=(r,d,\chi)\in \Z^3$ with $\mu:=\frac{d}{r}$ satisfying $-1<\mu \leq 0$, we have the following correspondence between geometric DT invariants $\overline{\Omega}_{v}^{+}$ of the local projective plane $K_{\PP^2}$ and the punctured Gromov--Witten invariants $N_{\tau,\beta}^{(X,D)}$ of the log Calabi--Yau surface $(X,D)$:
\[ \overline{\Omega}_{v}^{+} = \frac{|\iota_{\psi(\gamma(v))} \omega|}{|\gamma(v)|} 
N_{\tau_{\gamma(v)}^\theta, \beta_{\gamma(v)}^\theta}^{(X,D)} \,,\]
where $\theta$ is a point in the anti-attractor chamber $-\RR_{\geq 0} \iota_{\gamma(v)} \omega_Q$.
\end{theorem}

\begin{proof}
The result follows by combining \eqref{eq_01} and \eqref{eq_02}.
\end{proof}

\begin{remark}
One can check that the classes $v$ such that $\Omega_v^+=1$ are exactly the classes of the exceptional vector bundles on $\PP^2$, and that the corresponding curve classes $\beta^\theta_{\gamma(v)}$ on $(X,D)$ are exactly the curve classes coming from exceptional curves of toric models of $(X,D)$. Moreover, under the correspondence $v \mapsto \beta^\theta_{\gamma(v)}$, the mutations of exceptional vector bundles on $\PP^2$ \cite{DrP} correspond to the mutations of the toric models of $(X,D)$ \cite{GHKbirational}. 
\end{remark}

Theorem \ref{thm_localP2} is compatible with the heuristic picture of \cite[\S 7]{Bp2} describing DT invariants of  $K_{\PP^2}$ in terms of holomorphic curves via mirror symmetry and hyperk\"ahler rotation. 
For every phase $\psi \in \RR/2\pi\Z$, the DT invariants of $K_{\PP^2}$ counting stable objects of the derived category supported on the zero section, for stability conditions parametrized by the stringy K\"ahler moduli space and with central charge of phase $\psi +\frac{\pi}{2}$, are expected to be related to counts of $J_\psi$-holomorphic curves in a non-compact hyperk\"ahler manifold $(\mathcal{U},I,J,K)$, where $J_\psi$ is the complex structure $J_\psi:= (\cos \psi) J +(\sin \psi) K$.
For every phase $\psi$, the DT invariants of $K_{\PP^2}$ with central charge of phase $\psi+\frac{\pi}{2}$ are captured by the scattering diagram $\foD_\psi$ on the stringy K\"ahler moduli space defined in \cite{bousseau2022bps}, and one expects this scattering diagram to give a tropical description of $J_\psi$-holomorphic curves in $\mathcal{U}$.
Theorem \ref{thm_localP2} establishes this correspondence for $\psi=\frac{\pi}{2}$. Indeed, when the surface $X$ is obtained by blowing up three points whose sum is linearly equivalent to zero on the toric boundary divisor of $X_\Sigma$,
the holomorphic symplectic cluster variety $U$ is exactly the mirror of $(\PP^2,E)$, where $E$ is a smooth elliptic curve, and so, by \cite[\S 7.1]{Bp2}, is isomorphic to the complex manifold $(\mathcal{U}, J_{\frac{\pi}{2}})$.
Finally, note that the correspondence for $\psi=0$ has been established in \cite{Bp2}: in this case, the complex manifold $(\mathcal{U}, J_0)$ is isomorphic to the complement $\PP^2 \setminus E$ of a smooth elliptic curve $E$ in $\PP^2$. While Theorem \ref{thm_localP2} gives a description of DT counts of \emph{normalized} Gieseker semistable sheaves in terms of holomorphic curves in $(\mathcal{U}, J_{\frac{\pi}{2}})$, \cite{Bp2} gives a description of DT counts of arbitrary Gieseker semistable sheaves in terms of holomorphic curves in $(\mathcal{U}, J_{0})$. In particular, DT counts of \emph{normalized} Gieseker semistable sheaves admit two completely different descriptions in terms of holomorphic curves in the two very different complex manifolds $(\mathcal{U}, J_{\frac{\pi}{2}})$ and $(\mathcal{U}, J_0)$.

\begin{remark}
Using \cite{bousseau2020quantum}, one can refine Theorem \ref{thm_localP2} to a refined DT/higher genus Gromov--Witten correspondence.
\end{remark}

\subsection{Cubic surfaces} 
\label{sec_ex_cubic}

Let $Q$ be the octahedral quiver in Figure \ref{Fig:octaquiver}. This quiver is mutation equivalent to the elliptic Dynkin $D_4^{(1,1)}$ quiver.
Let $\mathbf{s}=(N,(e_i)_{1\leq i \leq 6}, \omega)$ be the symplectic  seed defined by $N=\Z^2$, 
\[ e_1=e_2=(1,0)\,, e_3=e_4=(0,1)\,, e_5=e_6=(-1,-1) \,,\]
and $\omega(-,-)=\det(-,-)$. Then, one checks that the map $\psi: N_Q \rightarrow N$ defined by $\psi(s_i)=e_i$ for all $1\leq i \leq 6$ is a compatibility data in the sense of Definition \ref{def_compatibility}.
In particular, this shows that $\omega_Q$ is of rank two.

\begin{figure}[h]
\center{\includegraphics{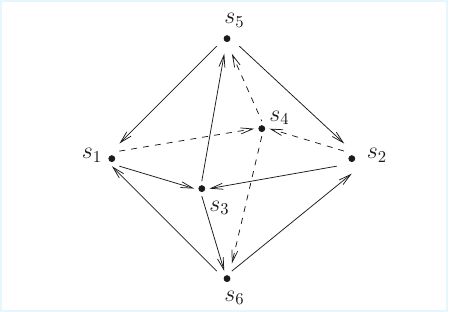}}
\caption{The quiver whose cluster variety is the cubic surface.}
\label{Fig:octaquiver}
\end{figure}

A log Calabi--Yau compactification $(X,D)$ of the corresponding cluster variety is obtained by blowing-up two general points on each of the the three toric divisors of $\PP^2$. In particular, $X$ is a cubic surface in $\PP^3$, $D$ is an anticanonical triangle of lines on $X$, and the holomorphic symplectic cluster surface $U=X \setminus D$ is an affine cubic surface. As the map $\psi$ is surjective, it follows that $U$ is isomorphic to a symplectic fiber of the Poisson $\mathcal{X}$ cluster variety defined by $Q$.

The quiver $Q$ can be obtained from an ideal triangulation of the 4-punctured sphere -- see for example \cite[Fig 9]{BY}.
In particular, as reviewed in Example \ref{ex_1}, it admits a quiver with potential $W$ \cite{LF} such that $(Q,W)$ has trivial attractor
DT invariants, and so we can apply Theorem \ref{thm_main}. As in \S \ref{sec_ex_local}, for every $\gamma \in N_Q \setminus \ker \omega_Q$ such that $\gamma \notin \Z_{\geq 1}s_i$ for all $1\leq i \leq 6$, the intersection $\gamma^\perp \cap M_\RR$ is the union of the attractor chamber $\RR_{\geq 0} \iota_\gamma \omega_Q$ and of the anti-attractor chamber $-\RR_{\geq 0} \iota_\gamma \omega_Q$.
If $\theta$ is in the attractor chamber, then $\overline{\Omega}_\gamma^\theta=0$. If $\theta$ is in the anti-attractor chamber, then $\mathcal{T}_\gamma^\theta$ contains a single wall $\tau_\gamma^\theta$, and so, by Theorem \ref{thm_main}, we have
\begin{equation} \label{eq_cubic}
\overline{\Omega}_\gamma^{+,\theta} = \frac{|\iota_{\psi(\gamma)}
\omega|}{|\gamma|} N_{\tau_{\gamma}^\theta, \beta_\gamma^\theta}^{(X,D)} \,.
\end{equation}
This example is particularly interesting because all the punctured Gromov--Witten invariants $N_{\tau,\beta}^{(X,D)}$ have been explicitly computed in \cite{GHKScubic} -- see also \cite{bousseau2023skein}. In particular, it follows that every non-zero DT invariant $\Omega_\gamma^{+,\theta}$ is equal to either one or two. 

Moreover, the collection of these DT invariants has a physics interpretation as the BPS spectrum of the 4-dimensional $\mathcal{N}=2$ supersymmetric $SU(2)$ gauge theory with $N_f=4$ flavors. The quiver with potential $(Q,W)$ is the BPS quiver of this theory -- see for example \cite[\S 4.7]{alim2014bps} and \cite{cecotti2013bps} where the two mutated versions of $Q$ are respectively mentioned. The BPS quiver has a natural origin from the class $\mathcal{S}$ construction of the gauge theory by compactification of the 6-dimensional $\mathcal{N}=(2,0)$ $A_1$ superconformal field theory on a 4-punctured sphere
\cite[\S 10.7]{GMN}, whereas the description of BPS states in terms of punctured Gromov--Witten invariants of $(X,D)$ arises from the realization of the gauge theory on the worldvolume of a M5-brane wrapping a Lagrangian torus in $U=X\setminus D$, as reviewed in \cite[\S 1.3]{bousseau2023skein}.

\bibliographystyle{plain}
\bibliography{bibliography}

\end{document}